\documentclass[12pt]{amsart}

\usepackage{graphicx}
\usepackage{subfigure}
\usepackage{eepic}
\usepackage{xcolor}
\selectcolormodel{gray}
\usepackage{tikz}
\usepackage{amsmath}
\usepackage{a4wide}
\usepackage[utf8]{inputenc}
\usepackage{amssymb}
\usepackage{amsopn}
\usepackage{epsfig}
\usepackage{amsfonts}
\usepackage{latexsym}
\usepackage{amsthm}
\usepackage{enumerate}
\usepackage[UKenglish]{babel}
\usepackage{verbatim}
\usepackage{color}
\newtheorem{theorem}{Theorem}[section]
\newtheorem{lemma}[theorem]{Lemma}
\newtheorem{proposition}[theorem]{Proposition}
\newtheorem{corollary}[theorem]{Corollary}

\theoremstyle{definition}

\theoremstyle{remark}

\numberwithin{equation}{section}

\begin{document}

\title{Exceptional digit frequencies and expansions in non-integer bases}

\author{Simon Baker}
\address{Mathematics institute, University of Warwick, Coventry, CV4 7AL, UK}
\email{simonbaker412@gmail.com}

\date{\today}


\begin{abstract}
In this paper we study the set of digit frequencies that are realised by elements of the set of $\beta$-expansions. The main result of this paper demonstrates that as $\beta$ approaches $1,$ the set of digit frequencies that occur amongst the set of $\beta$-expansions fills out the simplex. As an application of our main result, we obtain upper bounds for the local dimension of certain biased Bernoulli convolutions.
\end{abstract}

\maketitle

\section{Introduction}\label{sec:1}
Let $M\in\mathbb{N}$ and $\beta\in(1,M+1].$ For any $x\in I_{\beta,M}:=[0,\frac{M}{\beta-1}]$ there exists a sequence $(a_i)\in\{0,\ldots,M\}^{\mathbb{N}}$ such that $$x=\Pi_{\beta}(a_i):=\sum_{i=1}^{\infty}\frac{a_i}{\beta^i}.$$ We call such a sequence a $\beta$-expansion of $x$. Note that $x$ has a $\beta$-expansion if and only if $x\in I_{\beta,M}.$ Expansions of this type were pioneered in the papers of Parry \cite{Parry} and R\'{e}nyi \cite{Renyi}. When $\beta=M+1$ then we are in the familiar setting of integer base expansions, where every $x\in [0,1]$ has a unique $(M+1)$-expansion, apart from a countable set of points that have precisely two. However, when $\beta\in(1,M+1)$ the set of $\beta$-expansions can exhibit far more exotic behaviour. 

Given $M\in\mathbb{N}$, $\beta\in(1,M+1],$ and $x\in I_{\beta,M},$ we let $$\Sigma_{\beta,M}(x):=\Big\{(a_i)\in\{0,\ldots,M\}^{\mathbb{N}}:\sum_{i=1}^{\infty}\frac{a_i}{\beta^i}=x\Big\}.$$ In \cite{BakG} the author proved that for any $M\in\mathbb{N}$ there exists a critical constant $\mathcal{G}(M)$ satisfying 
\begin{equation*}
\mathcal{G}(M) = \left\{ \begin{array}{rl}
k+1 &\mbox{ if $M=2k$} \\
\frac{k+1+\sqrt{k^{2}+6k+5}}{2}   &\mbox{ if $M=2k+1$,}
\end{array} \right.
\end{equation*}  such that if $\beta\in(1,\mathcal{G}(M))$ then $\textrm{card}(\Sigma_{\beta,M}(x))=2^{\aleph_0}$ for every $x\in(0,\frac{M}{\beta-1})$. Moreover $\mathcal{G}(M)$ is optimal in the sense that if $\beta\in[\mathcal{G}(M),M+1],$ then there exists $x\in(0,\frac{M}{\beta-1})$ such that $\Sigma_{\beta,M}(x)$ is at most countable. Note that if $x$ is an endpoint of $I_{\beta,M}$ then $\Sigma_{\beta,M}(x)$ is always either $\{(0)^{\infty}\}$ or  $\{(M)^{\infty}\}$. As such, all of the interesting behaviour occurs within the interior of $I_{\beta,m}.$ For $\beta\in[\mathcal{G}(M),M+1)$ the cardinality of $\Sigma_{\beta,M}(x)$ for a generic $x$ is best described by a result of Sidorov, see \cite{Sid2,Sid3}. This result implies that for any $\beta\in[\mathcal{G}(M),M+1),$ we have $\textrm{card}(\Sigma_{\beta,M}(x))=2^{\aleph_0}$ for Lebesgue almost every $x\in I_{\beta,M}$. We remark that all other possible values of $\textrm{card}(\Sigma_{\beta,M}(x))$ are achievable. That is, for any $k\in \mathbb{N}\cup\{\aleph_0\},$ there exists $\beta\in(1,2)$ and $x\in(0,\frac{1}{\beta-1})$ such that $\textrm{card}(\Sigma_{\beta,1}(x)) =k,$ see \cite{Bak2,BakerSid,Sid1} and the references therein.

This paper is motivated by the following general question. Suppose we are interested in some property of a sequence $(a_i)\in\{0,\ldots,M\}^{\mathbb{N}}$. Properties we might be interested in could be combinatorial, number theoretic, or statistical. Can we put conditions on $\beta,$ such that every $x\in(0,\frac{M}{\beta-1})$ admits a sequence $(a_i)\in \Sigma_{\beta,M}(x)$ that satisfies this property? Alternatively, can we put conditions on $\beta,$ such that Lebesgue almost every $x\in(0,\frac{M}{\beta-1})$ admits a sequence $(a_i)\in \Sigma_{\beta,M}(x)$ that satisfies this property? Since an $x$ may well have infinitely many $\beta$-expansions, answering these questions is non-trivial. The general problem put forward here has been studied previously in different guises by several authors, see \cite{BakD,BakKong,DKK,Gun,HS,JSS}. In this paper we are interested in those sequences which exhibit exceptional digit frequencies. What exactly we mean by exceptional will become clear. 

The digit frequencies of a representation of a real number is a classical subject going back to the pioneering work of Borel \cite{Bor}, and later Besicovitch \cite{Bes} and Eggleston \cite{Egg}. Despite being a subject that has its origins in the early 20th century, representations of real numbers and their digit frequencies is still motivating researchers. For some recent contributions in this area see \cite{BCH,FLMW,HocShm} and the references therein. Most of the existing work in this area was done in a setting where the representation is unique. What distinguishes this work is that we are in a setting where the representations are almost certainly not unique.

\subsection{Statement of results}
Given $(a_i)\in\{0,\ldots,M\}^{\mathbb{N}}$ and $k\in\{0,\ldots,M\},$ we define the $k$-frequency of $(a_i)$ to be
$$\text{freq}_{k}(a_i):=\lim_{n\to\infty}\frac{\#\{1\leq i\leq n:a_i=k\}}{n}.$$ Assuming the limit exists. We say that $(a_i)$ is simply normal if for each $k\in\{0,\ldots,M\}$ the $k$-frequency exists and $\text{freq}_{k}(a_i)=\frac{1}{M+1}$. Borel's normal number theorem tells us that Lebesgue almost every $x$ has a simply normal $(M+1)$-expansion, see \cite{Bor}.

Combining the results of \cite{BakD} and \cite{BakKong} the following theorem is known to hold in the case where $M=1$.

\begin{theorem}
	\label{simple theorem}
The following statements hold.
\begin{enumerate}
	\item Let $\beta\in(1,1.80194\ldots]$. Then for every $x\in(0,\frac{1}{\beta-1})$ there exists $(a_i)\in \Sigma_{\beta,1}(x)$ such that $(a_i)$ is simply normal.
	\item Let $\beta\in(1,\frac{1+\sqrt{5}}{2})$. Then for every $x\in(0,\frac{1}{\beta-1})$ there exists $(a_i)\in \Sigma_{\beta,1}(x)$ such that the $0$-frequency of $(a_i)$ and the $1$-frequency of $(a_i)$ both don't exist.
	\item Let $\beta\in(1,\frac{1+\sqrt{5}}{2})$. Then there exists $c=c(\beta)>0$ such that for every $x\in (0,\frac{1}{\beta-1})$ and $p\in[1/2-c,1/2+c],$ there exists $(a_i)\in \Sigma_{\beta,1}(x)$ such that $\text{freq}_{0}(a_i)=p$ and $\text{freq}_{1}(a_i)=1-p.$
\end{enumerate}	
\end{theorem} In the above the number $1.80194\ldots$ is the unique root of $x^3-x^2-2x+1=0$ that lies within the interval $(1,2)$. Note that the intervals appearing in the three statements of Theorem \ref{simple theorem} are optimal. If $\beta\in(1.80194\ldots,2)$ then there exists an $x\in(0,\frac{1}{\beta-1})$ with no simply normal $\beta$-expansions. Similar statements hold for parts $(2)$ and $(3)$ of Theorem \ref{simple theorem}.

Theorem \ref{simple theorem} provides no information as to what frequencies are realised as $\beta$ approaches $1,$ or what happens for larger alphabets. Examining the techniques used in the proof of statement $(3)$ from Theorem \ref{simple theorem}, we see that they cannot improve upon the estimate $0<c(\beta)\leq 1/6$ for all $\beta\in(1,\frac{1+\sqrt{5}}{2})$. One might expect that as $\beta$ approaches $1,$ the set of realisable frequencies fills out the relevant simplex. In this paper we show this to be the case. We now express this formally. 

Let $$\Delta_{M}:=\Big\{(p_k)_{k=0}^M\in \mathbb{R}^{M+1}:p_k\geq 0,\, \sum_{k=0}^M p_k =1\Big\}.$$ We refer to an element of $\Delta_{M}$ as a frequency vector. Given an $\epsilon>0$ we define $$\Delta_{M,\epsilon}:=\Big\{(p_k)_{k=0}^{M}\in \mathbb{R}^{M+1}:0\leq p_k\leq 1- \epsilon,\, \sum_{k=0}^{M} p_k =1\Big\}.$$  Given $M\in\mathbb{N}$, $\beta\in(1,M+1]$, and $x\in I_{\beta,M},$ we let 
$$\Delta_{M,\beta}(x):=\Big\{(p_k)\in \Delta_M: \exists (a_i)\in \Sigma_{\beta,M}(x) \textrm{ such that } \text{freq}_k(a_i)=p_k\, \forall\, 0\leq k\leq M\Big\}.$$ 

We now introduce an important class of algebraic integers. To each $n\in\mathbb{N}$ we let $\beta_n\in(1,2)$ be the unique zero of $$f_{n}(x):=x^{n+1}-x^n-1$$ that is contained within the interval $(1,2)$. The fact that $\beta_n$ exists and is unique follows from the observations $f_{n}(1)=-1$, $f_{n}'(x)\geq 1$ for all $x\geq 1,$ and $f(2)>0.$ %Note also the trivial observation that if $\beta>1$ and $f_n(\beta)<0,$ then $\beta\in(1,\beta_n).$ 

In this paper we prove the following result. 

\begin{theorem}
\label{Main theorem}	
For any $M\in \mathbb{N}$ and $\beta\in(1,\beta_n),$  we have $$\Delta_{M,\frac{1}{n+1}}\subseteq \Delta_{M,\beta}(x)$$ for all $x\in(0,\frac{M}{\beta-1})$.
\end{theorem} Theorem \ref{Main theorem} demonstrates that as $\beta$ approaches $1,$ the set $\Delta_{M,\beta}(x)$ fills out the simplex $\Delta_{M}$ for any $x\in (0,\frac{M}{\beta-1})$. The following corollary of Theorem \ref{Main theorem} follows immediately.
\begin{corollary}
	\label{simple cor}
Let $M\in\mathbb{N}$ and $\beta\in(1,\frac{1+\sqrt{5}}{2}).$ Then every $x\in(0,\frac{M}{\beta-1})$ admits a simply normal $\beta$-expansion.
\end{corollary}
One might also ask whether one can construct $\beta$-expansions for which the digit frequencies do not exist. Theorem \ref{simple theorem} tells us that when $M=1$ and $\beta\in(1,\frac{1+\sqrt{5}}{2})$ every $x\in(0,\frac{1}{\beta-1})$ has a $\beta$-expansion for which the digit frequencies do not exist. The following theorem shows that the same phenomenon persists for larger alphabets.
\begin{theorem}
	\label{noexistence}
Let $M\geq 2$ and $\beta\in(1,\beta_n)$. Given $D\subseteq\{0,\ldots,M\}$ such that $\#D\geq 2,$ and $(p_{k})_{k\in D^c}$ such that $0\leq p_k\leq \frac{n}{n+1}$ for all $k\in D^c$ and $\sum_{k\in D^c}p_k<1$. Then for any $x\in(0,\frac{M}{\beta-1})$ there exists $(a_i)\in \Sigma_{\beta,M}(x)$ such that
\[ \text{freq}_k(a_i) = \left\{\begin{array}{ll}\textrm{does not exist} & \mbox{if $k \in D$};\\
p_k& \mbox{if $k\in D^c$}.\end{array} \right. \] 
\end{theorem}
Note that in Theorem \ref{noexistence} we cannot remove the condition $\#D\geq 2$, or the condition $\sum_{k\in D^c}p_k<1$. Removal of either of these conditions forces all of the digit frequencies to exist. Note that in Theorem \ref{noexistence} we could simply take $D=\{0,\ldots,M\}.$ As such we have the following corollary.

\begin{corollary}
Let $M\geq 2$ and $\beta\in(1,\frac{1+\sqrt{5}}{2}).$ Then every $x\in(0,\frac{M}{\beta-1})$ admits a $\beta$-expansion such that $\text{freq}_k(a_i)$ does not exist for all $0\leq k\leq M$.
\end{corollary}

It is natural to wonder how optimal the constant $\beta_n$ is. With this in mind we introduce the following. Given $M\in\mathbb{N}$ and $n\in\mathbb{N}$ let $$\beta_{M,n}^*:=\sup\Big\{\beta^*:\forall \beta\in(1,\beta^*)\textrm{ we have }\Delta_{M,\frac{1}{n+1}}\subseteq \Delta_{M,\beta}(x)\, \forall x\in\Big(0,\frac{M}{\beta-1}\Big)\Big\}.$$ By Theorem \ref{Main theorem} we know that $\beta_n\leq \beta_{M,n}^*$ for all $n\in\mathbb{N}$ and $M\in\mathbb{N}$. The following result describes the asymptotics of $\beta_n$ and $\beta_{M,n}^*$.
	
\begin{theorem}
\label{asymptotics} For any $M\in\mathbb{N}$ we have
	$$ 1+\frac{\log n -\log \log n}{n}< \beta_n\leq \beta_{M,n}^*\leq 1 + \frac{\log (n+1)}{n+1}+\mathcal{O}\Big(\frac{1+\log M}{n+1}\Big)$$
\end{theorem}In the above and throughout we make use of the standard big $\mathcal{O}$ notation. Theorem \ref{asymptotics} demonstrates that as $n$ tends to infinity, $\beta_n$ becomes a better approximation to the optimal value $\beta_{M,n}^*$. It is possible to obtain upper bounds for the quantity $\beta_{M,n}^*$ via existing results in \cite{BakG}, and by carefully examining the proof of Theorem \ref{asymptotics}. Omitting the relevant calculations we include in Figure $1$ a table of values detailing some upper bounds for $\beta_{2,n}^*$ along with some values for $\beta_n$.

The following corollary is an immediate consequence of Theorem \ref{asymptotics}.

\begin{corollary}
	\label{asymptotics cor}
For any $M\in\mathbb{N}$ we have $$(\beta_n-1)\sim(\beta_{M,n}^*-1)\sim \frac{\log n}{n}.$$
\end{corollary}
It is a surprising consequence of Corollary \ref{asymptotics cor} that the leading order term for the rate at which $\beta_{M,n}^*$ accumulates to $1$ has no dependence of $M$.

\begin{figure}
\centering
	\begin{tabular}{ |c|c|c|c| } 
		\hline
		n & $\beta_n$ & Upper bound for $\beta_{2,n}^*$ \\
	\hline
	 1 & $\frac{1+\sqrt{5}}{2}= 1.618\ldots$ & 2\\ 
		 2 & $1.466\ldots$ & $2$\\ 
		 3 & $1.380\ldots$ & $2$\\ 
		 4 & $1.325\ldots$ & $1.894\ldots$ \\
		 5 & $1.285\ldots$ & $1.761\ldots$\\
		 10& $1.184\ldots$ & $1.432\ldots$\\
		 25& $1.098\ldots$ & $1.207\ldots$\\
		 50& $1.058\ldots$ & $1.116\ldots$\\
		 100& $1.034\ldots$ & $1.064\ldots$\\
		\hline
	\end{tabular}
\caption{A table of values for $\beta_n$.}
\label{fig3}
\end{figure}

The rest of this paper is structured as follows. In Section $2$ we recall some useful dynamical preliminaries and prove some technical results that will be required later. In Section $3$ we prove Theorem \ref{Main theorem} and Theorem \ref{noexistence}. In Section $4$ we prove Theorem \ref{asymptotics}. We conclude in Section $5$ where we apply our results to obtain bounds on the local dimension of certain biased Bernoulli convolutions.

\section{Preliminaries}
We start by detailing a useful dynamical interpretation of $\Sigma_{\beta,M}(x)$. Given $\beta\in(1,M+1]$ and $k\in\{0,\ldots,M\},$ we introduce the map $T_{k}(x)=\beta x - k.$ Given an $x\in I_{\beta,M}$ we let $$\Omega_{\beta,M}(x):=\big\{(T_{a_i})\in\{T_0,\ldots, T_M\}^{\mathbb{N}}: (T_{a_j}\circ \cdots \circ T_{a_1})(x)\in I_{\beta,M}\, \forall j\in\mathbb{N}\big\}.$$ The following lemma was proved in \cite{Bak2}.
\begin{lemma}
	\label{Bijection lemma}
For any $x\in I_{\beta,M}$ we have $\textrm{card}(\Sigma_{\beta,M}(x))=\textrm{card}(\Omega_{\beta,M}(x)).$  Moreover, the map sending $(a_i)$ to $(T_{a_i})$ is a bijection between $\Sigma_{\beta,M}(x)$ and $\Omega_{\beta,M}(x).$ 
\end{lemma}Lemma \ref{Bijection lemma} allows us to reinterpret Theorem \ref{Main theorem} and Theorem \ref{noexistence} in terms of the existence of an element of $\Omega_{\beta,M}(x)$ exhibiting certain asymptotics. With this in mind we introduce the following notation. Given $\alpha=(\alpha_i)\in \cup_{j=1}^{\infty}\{T_0,\ldots,T_M\}^j$ we let $|\alpha|$ denote its length, and for any $k\in\{0,\ldots,M\}$ we let $$|\alpha|_k:=\#\{1 \leq i \leq |\alpha|:\alpha_i=T_{k}\}.$$ Given $\alpha=(\alpha_i)_{i=1}^j$ and $x\in I_{\beta,M},$ we let $\alpha(x):=(\alpha_j\circ \cdots \circ \alpha_1)(x)$. For notational convenience we let $\{T_0,\ldots,T_M\}^{0}$ denote the set consisting of the identity map. Given a finite sequence $\alpha\in\{T_0,\ldots,T_M\}^j,$ we let $\alpha^k$ denote its $k$-fold concatenation with itself, and let $\alpha^{\infty}$ denote the infinite concatenation of $\alpha$ with itself. We make use of analogous notational conventions for concatenations of finite sequences of digits.

It is straightforward to show that the unique fixed point of $T_{k}$ is $\frac{k}{\beta-1}$. It is also straightforward to show that the following equality holds for any $x\in\mathbb{R}$ and $l\in\mathbb{N}$ 
\begin{equation}
\label{expansion}
T_{k}^l(x)-\frac{k}{\beta-1}=\beta^l\Big(x-\frac{k}{\beta-1}\Big).
\end{equation} Despite being a simple observation, equation \eqref{expansion} will be extremely useful. We point out here three immediate consequences of it without proof. 

\begin{lemma}
	\label{easy lemma}
	The following properties hold.
\begin{enumerate}
	\item Let $\epsilon>0.$ There exists $L\in\mathbb{N}$ depending only upon $M$, $\beta$ and $\epsilon,$ such that for any $x>\frac{k}{\beta-1}+\epsilon$ and $y\in[x,\frac{M}{\beta-1}]$ we have $$T_{k}^{l}(x)\in [y,T_{k}(y)]$$ for some $1\leq l\leq L$.

	\item Let $\epsilon>0.$ There exists $L\in\mathbb{N}$ depending only upon $M$, $\beta$ and $\epsilon,$ such that for any  $x<\frac{k}{\beta-1}-\epsilon$ and $y\in[0, x]$ we have $$T_{k}^{l}(x)\in [T_{k}(y),y]$$ for some $1\leq l\leq L$.

	\item Let $\epsilon>0$. There exists $L\in\mathbb{N}$ depending only upon $M$, $\beta$ and $\epsilon$ such that if $k_1<k_2$ and  $x\in[\frac{k_1}{\beta-1}+\epsilon,\Pi_{\beta}((k_1,k_2)^{\infty})]$ then $$T_{k_1}^l(x)\in [\Pi_{\beta}((k_1,k_2)^{\infty}),\Pi_{\beta}((k_2,k_1)^{\infty})] $$ for some $1\leq l\leq L.$ Similarly, if $x\in[\Pi_{\beta}((k_2,k_1)^{\infty}),\frac{k_2}{\beta-1}-\epsilon]$ then $$T_{k_2}^l(x)\in [\Pi_{\beta}((k_1,k_2)^{\infty}),\Pi_{\beta}((k_2,k_1)^{\infty})] $$ for some $1\leq l\leq L.$ 
\end{enumerate}
	\end{lemma}Statement $(3)$ in Lemma \ref{easy lemma} follows from the previous two statements and the observations $$T_{k_1}(\Pi_{\beta}((k_1,k_2)^{\infty}))=\Pi_{\beta}((k_2,k_1)^{\infty})\textrm{ and } T_{k_2}(\Pi_{\beta}((k_2,k_1)^{\infty}))=\Pi_{\beta}((k_1,k_2)^{\infty}).$$
In our applications the $\epsilon$ appearing in Lemma \ref{easy lemma} will often be a function of $\beta.$ Therefore the relevant $L$ will often only depend upon $M$ and $\beta$. Despite being a fairly trivial observation, Lemma \ref{easy lemma} will be useful throughout. In particular statement $(3)$.

When constructing $\beta$-expansions that satisfy certain asymptotics, it is useful to partition the interval $I_{\beta,M}$ into subintervals for which we have a lot of control over how the different $T_{k}$ behave. This technique was originally used in \cite{BakG} and \cite{Bak3} to study the size of $\Sigma_{\beta,M}(x)$.

For any $\beta\in(1,M+1]$ and $k_1,k_2\in\{0,\ldots,M\}$ such that $k_1<k_2,$ we let $$I_{\beta,k_1,k_2}:=\Big[\frac{k_1}{\beta-1},\frac{k_2}{\beta-1}\Big].$$ Importantly, if $\beta\in(1,2]$ then every $x\in I_{\beta,k_1,k_2}$ admits a $\beta$-expansion whose digits are restricted to the set $\{k_1,k_2\}$. We also associate the following subinterval of $I_{\beta,k_1,k_2}$, let $$S_{\beta,k_1,k_2}:=[\Pi_{\beta}((k_2,k_1^{\infty})),\Pi_{\beta}((k_1,k_2^{\infty}))]=\Big[\frac{k_2}{\beta}+\frac{k_1}{\beta(\beta-1)},\frac{k_1}{\beta}+\frac{k_2}{\beta(\beta-1)}\Big].$$ For any $\beta\in(1,2)$ the interval $S_{\beta,k_1,k_2}$ is well defined and has non-empty interior. Note that $S_{\beta,k_1,k_2}$ is precisely the set of $x$ such that $T_{k_1}(x)\in I_{\beta,k_1,k_2}$ and $T_{k_2}(x)\in I_{\beta,k_1,k_2}.$ In the literature $S_{\beta,k_1,k_2}$ is commonly referred to as the switch region corresponding to $k_1$ and $k_2$.  We refer the reader to Figure $2$ for a diagram detailing the above.
%Its endpoints satisfy $$T_{k_2}\Big(\frac{k_2}{\beta}+\frac{k_1}{\beta(\beta-1)}\Big)=\frac{k_1}{\beta-1} \textrm{ and }T_{k_1}\Big(\frac{k_1}{\beta}+\frac{k_2}{\beta(\beta-1)}\Big)=\frac{k_2}{\beta-1}.$$ 

\begin{figure}
	\centering
	\begin{tikzpicture}[x=2.1,y=2.1]
	\path[draw](0,10) -- (0,160) -- (150,160) -- (150,10) -- (0,10);
	\path[draw][thick](0,10) -- (100,160);
	%\path[draw][dashed](100,160) -- (100,10);
	\path[draw][thick](50,10) -- (150,160);
	\path[draw](50,8) -- (50,12);
	\path[draw](100,8) -- (100,12);
	\draw (0,10) node[below] {$\frac{k_1}{\beta-1}$};
	\draw (50,7) node[below] {$\frac{k_2}{\beta}+\frac{k_1}{\beta(\beta-1)}$};
	\draw (100,7) node[below] {$\frac{k_1}{\beta}+\frac{k_2}{\beta(\beta-1)}$};
	\draw (150,10) node[below] {$\frac{k_2}{\beta-1}$};
	\draw (-5,165) node[below] {$\frac{k_2}{\beta-1}$};
	%\path(50,0)(150,150)
	%\path(50,-2)(50,2)
	%\path(100,-2)(100,2)
	\end{tikzpicture}
	\caption{The overlapping graphs of $T_{k_1}$ and $T_{k_2}$.}
	\label{fig1}
	
\end{figure}
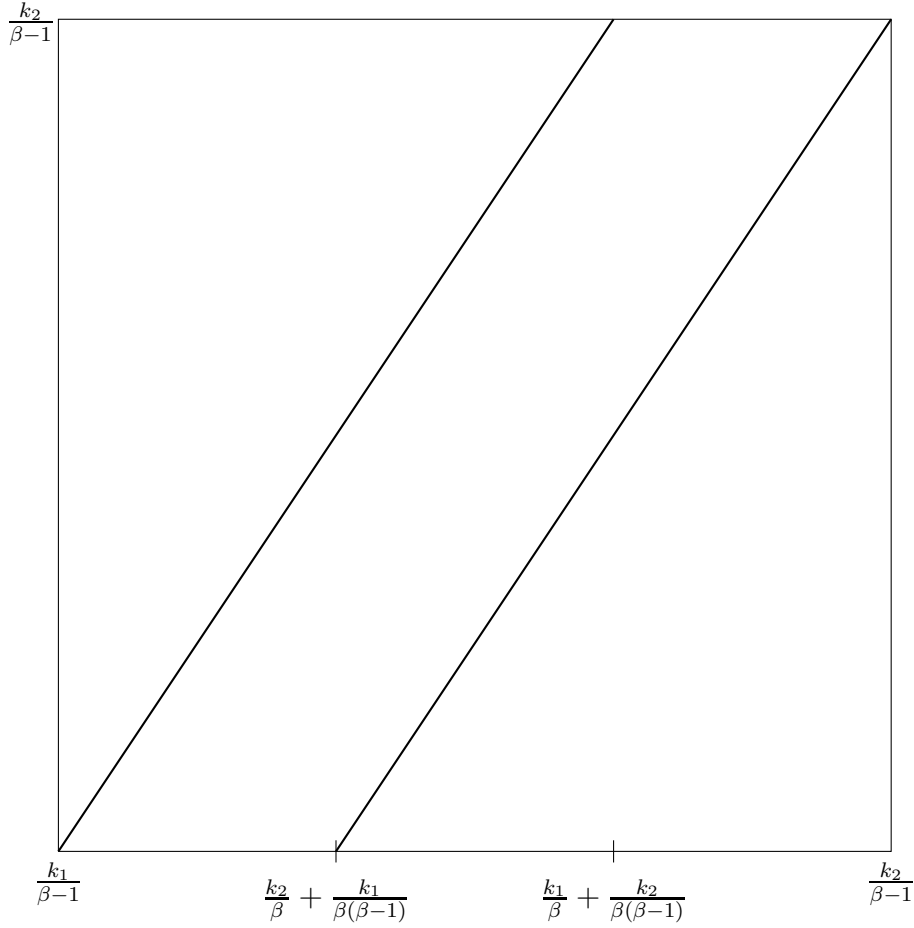

The following lemma proves that a useful class of subintervals depending on $n, k_1,$ and $k_2$, will always be contained in the switch region corresponding to $k_1$ and $k_2$ for $\beta\in(1,\beta_n).$
\begin{lemma}
	\label{inclusions}
For any $\beta\in(1,\beta_n)$ and $k_1,k_2\in\{0,\ldots,M\}$ such that $k_1<k_2$ we have 
\begin{equation}
\label{first inclusion}
[\Pi_{\beta}((k_1,k_2^n)^{\infty}),\Pi_{\beta}((k_2,k_1,k_2^{n-1})^{\infty})]	\subseteq \Big(\frac{k_2}{\beta}+\frac{k_1}{\beta(\beta-1)},\frac{k_1}{\beta}+\frac{k_2}{\beta(\beta-1)}\Big)
\end{equation}
and
\begin{equation} 
\label{second inclusion}
[\Pi_{\beta}((k_1,k_2,k_1^{n-1})^{\infty}),\Pi_{\beta}((k_2,k_1^n)^{\infty})]	\subseteq \Big(\frac{k_2}{\beta}+\frac{k_1}{\beta(\beta-1)},\frac{k_1}{\beta}+\frac{k_2}{\beta(\beta-1)}\Big).
\end{equation}
\end{lemma}

\begin{proof}
Fix $n\in\mathbb{N}$. We will only show that \eqref{first inclusion} holds. The proof that \eqref{second inclusion} holds follows similarly.

It suffices to show that for any $\beta\in(1,\beta_n)$ we have 
\begin{equation}
\label{check1}
\Pi_{\beta}((k_2,k_1,k_2^{n-1})^{\infty})<\frac{k_1}{\beta}+\frac{k_2}{\beta(\beta-1)}
\end{equation} and 
\begin{equation}
\label{check2}
\frac{k_2}{\beta}+\frac{k_1}{\beta(\beta-1)}<\Pi_{\beta}((k_1,k_2^n)^{\infty}).
\end{equation} 	

Rewriting the right hand side of \eqref{check1} using the geometric series formula we see that \eqref{check1} is equivalent to
$$\sum_{i=0}^{\infty}\frac{1}{\beta^{(n+1)i}}\Big(\frac{k_2}{\beta}+\frac{k_1}{\beta^2}+\frac{k_2}{\beta^3}+\cdots + \frac{k_2}{\beta^{n+1}}\Big)<\frac{k_1}{\beta}+\sum_{i=2}^{\infty}\frac{k_2}{\beta^i}.$$ Comparing coefficients, we see that this inequality is equivalent to $$\frac{k_2-k_1}{\beta}<\sum_{i=0}^{\infty}\frac{k_2-k_1}{\beta^{(n+1)i+2}}.$$ Cancelling coefficients and using the geometric series formula, we can rewrite this expression as
\begin{equation}
\label{nodep1}
\beta^{n+1}-\beta^{n}-1<0.
\end{equation}If $\beta\in(1,\beta_n)$ then \eqref{nodep1} is satisfied, and consequently \eqref{check1} is also satisfied.

We now turn our attention to \eqref{check2}. Expressing the left hand side of \eqref{check2} using the geometric series formula, we see that it is equivalent to 
$$\frac{k_2}{\beta}+\sum_{i=2}^{\infty}\frac{k_1}{\beta^i}<\sum_{i=0}^{\infty}\frac{1}{\beta^{i(n+1)}}\Big(\frac{k_1}{\beta}+\frac{k_2}{\beta^2}+\cdots + \frac{k_2}{\beta^{n+1}}\Big).$$ Comparing coefficients, we see that this is equivalent to $$\frac{k_2-k_1}{\beta}<\sum_{i=0}^{\infty}\frac{(k_2-k_1)}{\beta^{i(n+1)}}\Big(\frac{1}{\beta^2}+\cdots +\frac{1}{\beta^{n+1}}\Big).$$ Cancelling coefficients and using the geometric series formula, we can then rewrite this expression as
\begin{equation}
\label{check3}
1<\Big(\frac{1}{\beta}+\cdots + \frac{1}{\beta^n}\Big)\frac{\beta^{n+1}}{\beta^{n+1}-1}.
\end{equation}When $n=1$ we can verify that \eqref{check3} holds whenever $\beta\in(1,\frac{1+\sqrt{5}}{2}).$ For $n\geq 2$ it is clear that \eqref{check3} is implied by 
$$1<\frac{1}{\beta}+\frac{1}{\beta^2}.$$ This is true for any $\beta\in(1,\frac{1+\sqrt{5}}{2})$. Therefore \eqref{check2} holds for $\beta\in(1,\frac{1+\sqrt{5}}{2}).$ To complete the proof of our lemma one has to verify that $\beta_n\leq \frac{1+\sqrt{5}}{2}$ for any $n\geq 1$. However, this is a straightforward consequence of the definition of $\beta_n$.
\end{proof}
In what follows it is useful to define the following intervals, let $$D_{\beta,M,n}:=[T_{1}( \Pi_{\beta}((0,1,0^{n-1})^{\infty})),T_{M-1}(\Pi_{\beta}((M,M-1,M^{n-1})^{\infty}))].$$ It follows from Lemma \ref{inclusions} that for any $\beta\in(1,\beta_n)$ we have $D_{\beta,M,n}\subseteq (0,\frac{M}{\beta-1}).$ Similarly, given $k_1,k_2\in\{0,\ldots,M\}$ such that $k_1<k_2,$ let $$D_{\beta,k_1,k_2,n}:=[T_{k_2}( \Pi_{\beta}((k_1,k_2,k_1^{n-1})^{\infty})),T_{k_1}( \Pi_{\beta}((k_2,k_1,k_2^{n-1})^{\infty}))].$$ It follows from Lemma \ref{inclusions} that for any $\beta\in(1,\beta_n)$ we have $D_{\beta,k_1,k_2,n}\subseteq (\frac{k_1}{\beta-1},\frac{k_2}{\beta-1}).$ One can check that we also have $D_{\beta,k_1,k_2,n}\subseteq D_{\beta,M,n}$  for any $k_1,k_2\in\{0,\ldots,M\}$ such that $k_1<k_2$ and $\beta\in(1,\beta_n).$ Clearly $D_{\beta,0,M,n}= D_{\beta,M,n}$ so our notation may at first seem redundant. However it will be useful to adopt both conventions. As we will see, within the interval $D_{\beta,M,n}$ we shall apply many different $T_k$'s to construct a $\beta$-expansion with certain asymptotics. Importantly these $T_k$'s will never map the point in question out of $D_{\beta,M,n},$ so this interval can be thought of as the state space of our construction. When we want to emphasise the state space interpretation we will use the notation $D_{\beta,M,n}$. Within an interval $D_{\beta,k_1,k_2,n}$ we will be able to build an element of $\cup_{j=1}^{\infty}\{T_{k_1},T_{k_2}\}^j$ which exhibits certain asymptotics. When we want to emphasise the fact we are building a sequence consisting of $T_0$'s and $T_M$'s with certain asymptotics we will use the notation $D_{\beta,0,M,n}.$

The following lemma establishes the existence of an upper bound for the number of maps required to map a point from the interior of $I_{\beta,M}$ into the interior of an $I_{\beta,k_1,k_2}.$

\begin{lemma}
	\label{first step}
	Let $M\in\mathbb{N},$ $\beta\in(1,2),$ and $\delta>0$. There exists $L\in\mathbb{N}$ and $\epsilon>0$, such that for any $x\in [\delta,\frac{M}{\beta-1}-\delta]$ and $k_1<k_2,$ there exists $\eta\in\cup_{j=0}^{L}\{T_0,\ldots, T_M\}^j$ satisfying $$\eta(x)\in\Big[\frac{k_1}{\beta-1}+\epsilon,\frac{k_2}{\beta-1}-\epsilon\Big].$$
\end{lemma}
\begin{proof}We show that for each $k_1<k_2$ there exists $L_{k_1,k_2}\in\mathbb{N}$ and $\epsilon_{k_1,k_2}>0,$ such that for any $x\in [\delta,\frac{M}{\beta-1}-\delta]$ there exists $\eta\in\cup_{j=0}^{L_{k_1,k_2}}\{T_0,\ldots, T_M\}^j$ satisfying $$\eta(x)\in\Big[\frac{k_1}{\beta-1}+\epsilon_{k_1,k_2},\frac{k_2}{\beta-1}-\epsilon_{k_1,k_2}\Big].$$ Taking $L=\max_{k_1,k_2}L_{k_1,k_2}$ and $\epsilon=\min_{k_1,k_2}\epsilon_{k_1,k_2}$ completes the proof. We now proceed via a case analysis.
\\
	
\noindent \textbf{Case 1. $(k_1=0,\, k_2=M)$.} If $k_1=0$ and $k_2=M$ then let $\epsilon_{0,M}=\delta.$ We then have $x\in [\epsilon_{0,M},\frac{M}{\beta-1}-\epsilon_{0,M}]$ automatically by our hypothesis. So we can take $L_{0,M}=0$.
\\

\noindent \textbf{Case 2. $(k_1>0,\, k_2<M)$.} It is a straightforward calculation to prove that for any $\beta\in(1,2)$ and $k_1<k_2$ we have
 \begin{equation}
\label{inclusiona}
[\Pi_{\beta}((k_1-1,k_2+1)^{\infty}),\Pi_{\beta}((k_2+1,k_1-1)^{\infty})]\subseteq \Big(\frac{k_1}{\beta-1},\frac{k_2}{\beta-1}\Big).
\end{equation} 
It therefore suffices to show that there exists $L_{k_1,k_2}\in\mathbb{N},$ such that for any $x\in [\delta,\frac{M}{\beta-1}-\delta]$ there exists $\eta \in\cup_{j=0}^{L_{k_1,k_2}}\{T_0,\ldots, T_M\}^j$ satisfying $$\eta(x)\in [\Pi_{\beta}((k_1-1,k_2+1)^{\infty}),\Pi_{\beta}((k_2+1,k_1-1)^{\infty})].$$ The existence of the desired $\epsilon_{k_1,k_2}$ will then follow by \eqref{inclusiona}.

Let us now fix $x\in [\delta,\frac{M}{\beta-1}-\delta]$ and $k_1,k_2$ such that $0<k_1<k_2<M.$ If $$x\in \big[ \Pi_{\beta}((k_1-1,k_2+1)^{\infty}),\Pi_{\beta}((k_2+1,k_1-1)^{\infty})]$$ then there is nothing to prove. Suppose this is not the case and $x<\Pi_{\beta}((k_1-1,k_2+1)^{\infty}),$ the case when $x>\Pi_{\beta}((k_2+1,k_1-1)^{\infty})$ is handled similarly. If $x<\Pi_{\beta}((k_1-1,k_2+1)^{\infty})$ then $$x\in\Big[\delta,\frac{1}{\beta-1}\Big],\, x\in \bigcup_{j=1}^{k_1-1}\Big[\frac{j}{\beta-1},\frac{j+1}{\beta-1}\Big],\, \textrm{ or }x\in \Big[\frac{k_1}{\beta-1},\Pi_{\beta}((k_1-1,k_2+1)^{\infty})\Big].$$ If $x\in [\delta,\frac{1}{\beta-1}],$ then by Lemma \ref{easy lemma} there exists $L_1\in\mathbb{N}$ such that $T_0^{l_1}(x)\in [\frac{1}{\beta-1},\frac{2}{\beta-1}]$ for some $l_1\leq L_1.$ Consequently, it suffices to consider the case where $$x\in \bigcup_{j=1}^{k_1-1}\Big[\frac{j}{\beta-1},\frac{j+1}{\beta-1}\Big]\, \textrm{ or } x\in\Big[\frac{k_1}{\beta-1},\Pi_{\beta}((k_1-1,k_2+1)^{\infty})\Big].$$ If $x\in \Big[\frac{j}{\beta-1},\frac{j+1}{\beta-1}\Big]$ for some $1\leq j \leq k_1-1,$ then by Lemma \ref{easy lemma} and \eqref{inclusiona}, there exists $L_2\in\mathbb{N}$ such that
\begin{equation}
\label{inclusionc}
T_{j-1}^{l_2}(x)\in [\Pi_{\beta}((j-1,j+2)^{\infty}),\Pi_{\beta}((j+2,j-1)^{\infty})]
\end{equation} for some $l_2\leq L_2.$ Lemma \ref{easy lemma}, \eqref{inclusiona}, and \eqref{inclusionc}, when combined then imply that there exists $L_3\in\mathbb{N}$, such that for some $l_3\leq L_3$ we have 
\begin{equation}
\label{inclusiond}
(T^{l_3}_{j}\circ T_{j_1-1}^{l_2})(x)\in [\Pi_{\beta}((j,j+3)^{\infty}),\Pi_{\beta}((j+3,j)^{\infty})].
\end{equation} We observe from equation \eqref{inclusiond} that $x$ has been mapped from $I_{\beta,j,j+1}$ into $[\Pi_{\beta}((j,j+3)^{\infty}),\Pi_{\beta}((j+3,j)^{\infty})]$. If $j=k_1-1$ then our proof is complete. If $j<k_1-1,$ then we can repeat the steps used to derive \eqref{inclusiond} from \eqref{inclusionc} and map the orbit of $x$ into the interval $[\Pi_{\beta}((j+1,j+4)^{\infty}),\Pi_{\beta}((j+4,j+1)^{\infty})]$ using only a bounded number of maps. Repeatedly applying this procedure we see that the orbit of $x$ must eventually be mapped into $[ \Pi_{\beta}((k_1-1,k_2+1)^{\infty}),\Pi_{\beta}((k_2+1,k_1-1)^{\infty})]$ by some $\eta$ as required. Moreover, since the number of maps needed to map a point from $[\Pi_{\beta}((j-1,j+2)^{\infty}),\Pi_{\beta}((j+2,j-1)^{\infty})]$ into $[\Pi_{\beta}((j,j+3)^{\infty}),\Pi_{\beta}((j+3,j)^{\infty})]$ can always be bounded above by some constant, it follows that length of $\eta$ can always be bounded above by some $L_{k_1,k_2}$.

If $$x\in \Big[\frac{k_1}{\beta-1},\Pi_{\beta}((k_1-1,k_2+1)^{\infty})\Big]$$ then $x$ is mapped into  $[ \Pi_{\beta}((k_1-1,k_2+1)^{\infty}),\Pi_{\beta}((k_2+1,k_1-1)^{\infty})]$ by repeatedly applying $T_{k_1-1}.$ We can bound the number of maps required to do this by Lemma \ref{easy lemma}.
\\

\noindent \textbf{Case 3. $(k_1>0,\, k_2=M)$.} 
For any $\beta>1$ we always have the inclusion $$[\Pi_{\beta}((k_1,M)^{\infty}),\Pi_{\beta}((M,k_1)^{\infty})]\subseteq \Big(\frac{k_1}{\beta-1},\frac{M}{\beta-1}\Big).$$ Therefore it suffices to show that there exists $L_{k_1,M}\in\mathbb{N},$ such that for any $x\in [\delta,\frac{M}{\beta-1}-\delta]$ there exists $\eta \in\cup_{j=0}^{L_{k_1,M}}\{T_0,\ldots, T_M\}^j$  satisfying $$\eta(x)\in[\Pi_{\beta}((k_1,M)^{\infty}),\Pi_{\beta}((M,k_1)^{\infty})].$$Fix $x\in [\delta,\frac{M}{\beta-1}-\delta]$. If $x\in [\Pi_{\beta}((k_1,M)^{\infty}),\Pi_{\beta}((M,k_1)^{\infty})]$ then there is nothing to prove. If $x\in(\Pi_{\beta}((M,k_1)^{\infty}),\frac{M}{\beta-1}-\delta],$ then by Lemma \ref{easy lemma} it follows that there exists $L_1\in\mathbb{N},$ such that $$T_{M}^{l_1}(x)\in [\Pi_{\beta}((k_1,M)^{\infty}),\Pi_{\beta}((M,k_1)^{\infty})]$$ for some $l_1\leq L_1$. Alternatively, if $x<\Pi_{\beta}((k_1,M)^{\infty})$ then one can replicate the argument used in Case $2$ to deduce that $x$ can be mapped using only a bounded number of maps into the interval $$\Big[\frac{k_1-1}{\beta-1}+\epsilon_{k_1-1,k_1},\Pi_{\beta}((k_1,M)^{\infty})\Big]$$ for some $\epsilon_{k_1-1,k_1}>0.$ Consequently, without loss of generality we may assume that $$x\in\Big[\frac{k_1-1}{\beta-1}+\epsilon_{k_1-1,k_1},\Pi_{\beta}((k_1,M)^{\infty})\Big].$$ At this point we make a simple observation. For any $\beta\in(1,2)$ and $0<k_1<M$ we have $$T_{k_1-1}\Big(\frac{k_1}{\beta-1}\Big)<\frac{M}{\beta-1}.$$ Therefore, by continuity there exists parameters $\gamma>0$ and $\kappa>0$ such that if 
\begin{equation}
\label{observationa}
x\in\Big[\frac{k_1}{\beta-1},\frac{k_1}{\beta-1}+\gamma\Big]\textrm{ then } T_{k_1-1}(x)\in\Big[\frac{k_1}{\beta-1}+\beta\Big(\frac{k_1}{\beta-1}-\frac{k_1-1}{\beta-1}\Big),\frac{M}{\beta-1}-\kappa\Big].
\end{equation} For $x\in [\frac{k_1-1}{\beta-1}+\epsilon_{k_1-1,k_1},\Pi_{\beta}((k_1,M)^{\infty})]$ it follows from Lemma \ref{easy lemma} that there exists $L_2\in\mathbb{N},$ such that the unique minimal $l_2\geq 0$ such that $T_{k_1-1}^{l_2}(x)\in[\frac{k_1}{\beta-1},\frac{M}{\beta-1}]$ is bounded above by $L_2$. 

If $l_2=0$ and $x\in[\frac{k_1}{\beta-1}+\gamma,\Pi_{\beta}((k_1,M)^{\infty})],$ then $x$ is a uniformly bounded distance away from the fixed point $\frac{k_1}{\beta-1}$. Therefore, by Lemma \ref{easy lemma} our point $x$ can be mapped into $[\Pi_{\beta}((k_1,M)^{\infty}),\Pi_{\beta}((M,k_1-1)^{\infty})]$ by a bounded number of maps. If $x\in [\frac{k_1}{\beta-1},\frac{k_1}{\beta-1}+\gamma]$ then we apply $T_{k_1-1}$ to $x$. By \eqref{observationa} the point $T_{k_1-1}(x)$ is a uniformly bounded distance away from the endpoints of $I_{\beta,k_1,M}.$ Therefore by Lemma \ref{easy lemma} we can bound the number of maps required to map $T_{k_1-1}(x)$ into $[\Pi_{\beta}((k_1,M)^{\infty}),\Pi_{\beta}((M,k_1)^{\infty})].$ 

If $l_2\geq 1$ then either $$T_{k_1-1}^{l_2}(x)\in \Big[\frac{k_1}{\beta-1},\frac{k_{1}}{\beta-1}+\gamma\Big] \textrm{ or }T_{k_1-1}^{l_2}(x)\in\Big(\frac{k_{1}}{\beta-1}+\gamma,\frac{k_1}{\beta-1}+\beta\Big(\frac{k_1}{\beta-1}-\frac{k_1-1}{\beta-1}\Big)\Big).$$ In the first case we proceed as in the case where $l_2=0$ and apply $T_{k_1-1}$ once more to ensure $T_{k_1-1}^{l_2+1}(x)$ is a bounded distance from the endpoints of $I_{\beta,k_1,M}.$ In the second case we are already a uniformly bounded distance away from the endpoints of $I_{\beta,k_1,M}.$ It follows from Lemma \ref{easy lemma} that $T_{k_1-1}^{l_2}(x)$ can be mapped into $[\Pi_{\beta}((k_1,M)^{\infty}),\Pi_{\beta}((M,k_1)^{\infty})]$ using a bounded number of maps.
\\

\noindent \textbf{Case 4. $(k_1=0,k_2<M)$.} The proof of Case $4$ is analogous to the the proof of Case $3$.

\end{proof}

The following lemma tells us that if $x$ is in the interior of an $I_{\beta,k_1,k_2}$ and a bounded distance from its endpoints, then we can map $x$ into the intervals appearing in Lemma \ref{inclusions} using a bounded number of maps.

\begin{lemma}
	\label{second step}
Let $M\in\mathbb{N},$ $\beta\in(1,\beta_n)$ and $\epsilon>0$. There exists $L\in\mathbb{N},$ such that for any $k_1<k_2$ and $x\in [\frac{k_1}{\beta-1}+\epsilon,\frac{k_2}{\beta-1}-\epsilon],$ there exists $\eta^1\in\cup_{j=0}^{L}\{T_{k_1},T_{k_2}\}^j$ satisfying
\begin{equation} 
\label{holds1}
\eta^1(x)\in  [\Pi_{\beta}((k_1,k_2^n)^{\infty}),\Pi_{\beta}((k_2,k_1,k_2^{n-1})^{\infty})].
\end{equation} Similarly, there exists $\eta^2\in\cup_{j=0}^{L}\{T_{k_1}, T_{k_2}\}^j$ satisfying
\begin{equation}
\label{holds2}
\eta^2(x)\in [\Pi_{\beta}((k_1,k_2,k_1^{n-1})^{\infty}),\Pi_{\beta}((k_2,k_1^n)^{\infty})].
\end{equation}
\end{lemma}

\begin{proof}
	Let us start by fixing $k_1< k_2$ and $x\in [\frac{k_1}{\beta-1}+\epsilon,\frac{k_2}{\beta-1}-\epsilon].$ We will only show that there exists $L_{k_1,k_2}\in\mathbb{N}$ and $\eta^1\in\cup_{j=0}^{L_{k_1,k_2}}\{T_{k_1},T_{k_2}\}^{j}$ such that \eqref{holds1} holds. The existence of an $L_{k_1,k_2}$ and an $\eta^2\in\cup_{j=0}^{L}\{T_{k_1}, T_{k_2}\}^{L}$ such that \eqref{holds2} holds follows by a similar argument. To finish the proof of the lemma we take $L=\max_{k_1,k_2} L_{k_1,k_2}.$

 If $$x\in [\Pi_{\beta}((k_1,k_2^n)^{\infty}),\Pi_{\beta}((k_2,k_1,k_2^{n-1})^{\infty})]$$ there is nothing to prove. Suppose $x>\Pi_{\beta}((k_2,k_1,k_2^{n-1})^{\infty}),$ then repeated iteration of $T_{k_2}$ eventually maps $x$ into $[\Pi_{\beta}((k_1,k_2^n)^{\infty}),\Pi_{\beta}((k_2,k_1,k_2^{n-1})^{\infty})]$ by Lemma \ref{easy lemma}. The fact the number of iterations of $T_{k_2}$ required to do this is bounded also follows from Lemma \ref{easy lemma} and the fact $x<\frac{k_2}{\beta-1}-\epsilon$. If $x<\Pi_{\beta}((k_1,k_2^n)^{\infty})$ then repeated iteration of $T_{k_1}$ maps $x$ into $[\Pi_{\beta}((k_1,k_2^n)^{\infty}),T_{k_1}(\Pi_{\beta}((k_1,k_2^n)^{\infty}))].$ Lemma \ref{easy lemma} and the fact $x>\frac{k_1}{\beta-1}+\epsilon$ implies that the number of maps required to map $x$ into $[\Pi_{\beta}((k_1,k_2^n)^{\infty}),T_{k_1}(\Pi_{\beta}((k_1,k_2^n)^{\infty}))]$ can be bounded above. If $x$ has been mapped into $[\Pi_{\beta}((k_1,k_2^n)^{\infty}),\Pi_{\beta}((k_2,k_1,k_2^{n-1})^{\infty})]$ then we are done. If not then $x$ has been mapped into $(\Pi_{\beta}((k_2,k_1,k_2^{n-1})^{\infty}),T_{k_1}(\Pi_{\beta}((k_1,k_2^n)^{\infty}))].$ Since $\beta\in(1,\beta_n)$ we know by Lemma \ref{inclusions} that $$\Pi_{\beta}((k_1,k_2^n)^{\infty})\in \Big(\frac{k_2}{\beta}+\frac{k_1}{\beta(\beta-1)},\frac{k_1}{\beta}+\frac{k_2}{\beta(\beta-1)}\Big).$$ Therefore $T_{k_1}(\Pi_{\beta}((k_1,k_2^n)^{\infty}))$ is some uniformly bounded distance away from $\frac{k_2}{\beta-1}$. Consequently, the image of $x$ within $(\Pi_{\beta}((k_2,k_1,k_2^{n-1})^{\infty}),T_{k_1}(\Pi_{\beta}((k_1,k_2^n)^{\infty}))]$ is some uniformly bounded distance away from $\frac{k_2}{\beta-1}.$ We now repeat our initial argument in the case where $x>\Pi_{\beta}((k_2,k_1,k_2^{n-1})^{\infty})$ to complete our proof.
\end{proof}
The following lemma follows from the proof of Lemma \ref{second step}. It poses greater restrictions on the orbit of $x$ under $\eta$.

\begin{lemma}
\label{third step}
Let $M\in\mathbb{N}$ and $\beta\in(1,\beta_n)$. Then there exists $L\in\mathbb{N}$, such that for any $k_1<k_2$ and $x\in D_{\beta,k_1,k_2,n}$, there exists $\eta^1\in \cup_{j=0}^L\{T_{k_1},T_{k_2}\}$ satisfying 
$$\eta^1(x)\in  [\Pi_{\beta}((k_1,k_2^n)^{\infty}),\Pi_{\beta}((k_2,k_1,k_2^{n-1})^{\infty})]$$
 and $$(\eta_{j}^1\circ \cdots \circ \eta_1^1)(x)\in D_{\beta,k_1,k_2,n}\textrm{ for all }1\leq j\leq |\eta^1|.$$ Similarly, there exists $\eta^2\in\cup_{j=0}^{L}\{T_{k_1}, T_{k_2}\}^j$ satisfying 
$$\eta^2(x)\in [\Pi_{\beta}((k_1,k_2,k_1^{n-1})^{\infty}),\Pi_{\beta}((k_2,k_1^n)^{\infty})]$$ and $$(\eta_{j}^2\circ \cdots \circ \eta_{1}^{2})(x)\in D_{\beta,k_1,k_2,n}\textrm{ for all }1\leq j\leq |\eta^2|.$$
\end{lemma}
Lemma \ref{third step} gives conditions ensuring that the orbit of $x$ under $\eta$ stays within the interval $D_{\beta,k_1,k_2,n}$. This property will be useful when we want our orbit to be mapped into yet another $D_{\beta,k_1',k_2',n}.$

The following lemma shows that if $x$ is contained in $D_{\beta,M,n},$ then $x$ can be mapped into the intervals appearing in Lemma \ref{inclusions} using a bounded number of maps.

\begin{lemma}
	\label{free movement}
Let $M\in\mathbb{N}$ and $\beta\in(1,\beta_n).$ Then there exists $L\in\mathbb{N}$, such that for any $x\in D_{\beta,M,n}$ and $k_1,k_2\in\{0,\ldots,M\}$ satisfying $k_1<k_2$, there exists  $\eta^1\in\cup_{j=0}^{L}\{T_0,\ldots, T_M\}^j$ satisfying
$$\eta^1(x)\in  [\Pi_{\beta}((k_1,k_2^n)^{\infty}),\Pi_{\beta}((k_2,k_1,k_2^{n-1})^{\infty})].$$ Similarly, there exists $\eta^2\in\cup_{j=0}^{L}\{T_0,\ldots, T_M\}^j$ satisfying
$$\eta^2(x)\in [\Pi_{\beta}((k_1,k_2,k_1^{n-1})^{\infty}),\Pi_{\beta}((k_2,k_1^n)^{\infty})].$$
\end{lemma}
\begin{proof}
	Lemma \ref{free movement} follows almost immediately from Lemma \ref{first step} and Lemma \ref{second step}. We include the proof for completion. Let us start by emphasising $D_{\beta,M,n}\subseteq (0,\frac{M}{\beta-1})$ for $\beta\in(1,\beta_n)$ and so is contained in $[\delta,\frac{M}{\beta-1}-\delta]$ for some $\delta$ depending on $M$ and $\beta$. Now fix $x\in D_{\beta,M,n}.$ By Lemma \ref{first step} there exists a bounded number of transformations that map $x$ into $[\frac{k_1}{\beta-1}+\epsilon,\frac{k_2}{\beta-1}-\epsilon]$ for some $\epsilon>0$. Applying Lemma \ref{second step} to the image of $x$ within $[\frac{k_1}{\beta-1}+\epsilon,\frac{k_2}{\beta-1}-\epsilon]$ allows us to assert that there exists a bounded number of maps that map this image of $x$ into $[\Pi_{\beta}((k_1,k_2^n)^{\infty}),\Pi_{\beta}((k_2,k_1,k_2^{n-1})^{\infty})]$. Hence our $\eta^1$ exists. The existence of $\eta^2$ follows from an analogous argument.	
\end{proof}

\section{Proofs of Theorems \ref{Main theorem} and \ref{noexistence}}
We now proceed with our proof of Theorem \ref{Main theorem}. Our proof relies on the following two propositions. 

\begin{proposition}
	\label{important prop}
	Let $M\in \mathbb{N}$ and $\beta\in(1,\beta_n).$ There exists $C>0$ such that, for any $k_1<k_2,$ $x\in D_{\beta,k_1,k_2,n},$ and $\textbf{p}=(p_{k_1},p_{k_2})\in \Delta_{1,\frac{1}{n+1}},$ there exists $\tau\in \{T_{k_1},T_{k_2}\}^{\infty}$ satisfying 
\begin{enumerate}
	\item $(\tau_N\circ \cdots \circ \tau_1)(x)\in D_{\beta,k_1,k_2,n}$ for all $N\in\mathbb{N},$
	\item $\Big| |(\tau_i)_{i=1}^N|_{k_1}-p_{k_1}N\Big|\leq C \textrm{ for all }N\in\mathbb{N},$
	\item $	\Big| |(\tau_i)_{i=1}^N|_{k_2}-p_{k_2}N\Big|\leq C \textrm{ for all }N\in\mathbb{N}.$
\end{enumerate}

\end{proposition}
\begin{proof}
The proof of Proposition \ref{important prop} relies on devising an algorithm that yields the desired $\tau$. At each step in the algorithm we should check rigorously that properties $(1), (2),$ and $(3)$ hold. However, for the sake of brevity we simply state here that property $(1)$ will hold since $\tau$ will be constructed by concatenating maps of the form guaranteed by Lemma \ref{third step}, maps from $[\Pi_{\beta}((k_1,k_2^n)^{\infty}),\Pi_{\beta}((k_2,k_1,k_2^{n-1})^{\infty})]$ to itself of the form $T_{k_2}^l\circ T_{k_1}$, and maps from $[\Pi_{\beta}((k_1,k_2,k_1^{n-1})^{\infty}),\Pi_{\beta}((k_2,k_1^n)^{\infty})]$ to itself of the form $T_{k_1}^l\circ T_{k_2}$. Also, note that since $\tau$ will be an element of $\{T_{k_1},T_{k_2}\}^{\mathbb{N}},$ property $(2)$ is equivalent to property $(3)$. So it suffices to prove property $(2)$.
	
By Lemma \ref{third step} we can assume without loss of generality that $$x\in [\Pi_{\beta}((k_1,k_2,k_1^{n-1})^{\infty}),\Pi_{\beta}((k_2,k_1^n)^{\infty})].$$  
\\
\noindent \textbf{Step 1}. For $x\in [\Pi_{\beta}((k_1,k_2,k_1^{n-1})^{\infty}),\Pi_{\beta}((k_2,k_1^n)^{\infty})]$ we have $$(T_{k_1}^{l_1}\circ T_{k_2})(x)\in [\Pi_{\beta}((k_1,k_2,k_1^{n-1})^{\infty}),\Pi_{\beta}((k_2,k_1^n)^{\infty})]$$ for some $n\leq l_1\leq n'$ by Lemma \ref{easy lemma} and Lemma \ref{inclusions}. Importantly $n'$ depends solely upon $M$ and $\beta$. Let $\tau^1=(T_{k_2},T_{k_1}^{l_1}).$ Note that
\begin{equation*}
\Big| |(\tau_i^1)_{i=1}^N|_{k_1}-p_{k_1}N\Big|\leq  n'+1\textrm{ for all }1\leq N\leq |\tau^1|.
\end{equation*} 

\noindent \textbf{Step 2}. At this point we remark that
\begin{equation}
\label{signcheck1}
|\tau^1|_{k_1}-p_{k_1}|\tau^1|=l_1-p_{k_1}(l_1+1)\geq 0.  
\end{equation} This is because $p_{k_1}\leq \frac{n}{n+1}$ and $l_1\geq n$. We now apply Lemma \ref{third step} to $\tau^1(x)$. There exists $L\in\mathbb{N}$ and $\eta^1\in \cup_{j=0}^L\{T_{k_1},T_{k_2}\}^j,$ such that 
\begin{equation}
\label{inclusiondumb2}
\eta^1(\tau^1(x))\in[\Pi_{\beta}((k_1,k_2^n)^{\infty}),\Pi_{\beta}((k_2,k_1,k_2^{n-1})^{\infty}),].
\end{equation} We let $\kappa^1=(\tau^1,\eta^1)$ and observe 
\begin{equation}
\label{checkc}
\Big| |(\kappa_i^1)_{i=1}^N|_{k_1}-p_{k_1}N\Big|\leq  n'+L+1\textrm{ for all }1\leq N\leq |\kappa^1|.
\end{equation}  At this point we examine the sign of 
\begin{equation}
\label{signcheck2}
|\kappa^1|_{k_1}-p_{k_1}|\kappa^1|.  
\end{equation} If \eqref{signcheck2} is negative we stop and let $\tau^2=\kappa^1$. Note that if this is the case then by \eqref{signcheck1} we have $$-L\leq |\tau^2|_{k_1}-p_{k_1}|\tau^2|\leq 0.$$  Suppose \eqref{signcheck2} is positive, we then use \eqref{inclusiondumb2} to assert that there exists $n\leq l_2\leq n'$ such that  $$(T_{k_2}^{l_2}\circ T_{k_1})(\kappa^1(x))\in [\Pi_{\beta}((k_1,k_2^n)^{\infty}),\Pi_{\beta}((k_2,k_1,k_2^{n-1})^{\infty})].$$  Where $n'$ is as above and depends only upon $M$ and $\beta$. Letting $\kappa^2=(\kappa^1,T_{k_1},T_{k_2}^{l_2})$ we see that \eqref{checkc} implies
\begin{equation}
\label{checke}
\Big| |(\kappa_i^2)_{i=1}^N|_{k_1}-p_{k_1}N\Big|\leq  2n'+L+2\textrm{ for all }1\leq N\leq |\kappa^2|.
\end{equation} We trivially have $$|\kappa^2|_{k_1}-p_{k_1}|\kappa_2|=|\kappa^1|_{k_1}-p_{k_1}|\kappa_1|+(1-p_{k_1}(l_2+1)).$$ 
Since $p_{k_1}\geq \frac{1}{n+1}$ and $l_2\geq n$ we have 
\begin{equation}
\label{increase}
|\kappa^2|_{k_1}-p_{k_1}|\kappa^2|\leq |\kappa^1|_{k_1}-p_{k_1}|\kappa_1|.
\end{equation} Equation \eqref{increase} when combined with \eqref{checkc} and the assumption \eqref{signcheck2} is positive implies
\begin{equation}
\label{level2 bound}
-n'\leq |\kappa^2|_{k_1}-p_{k_1}|\kappa_2|\leq n'+L+1.
\end{equation} At this point we examine the sign of
\begin{equation}
\label{signcheck3}
|\kappa^2|_{k_1}-p_{k_1}|\kappa^2|.
\end{equation}If \eqref{signcheck3} is negative we stop and let $\tau^2=\kappa^2.$ In which case 
$$-n'\leq |\tau^2|_{k_1}-p_{k_1}|\tau^2|\leq 0.$$ If \eqref{signcheck3} is positive then we repeat our previous step with $\kappa^1(x)$ replaced with $\kappa^2(x)$. So there exists $n\leq l_3\leq n'$ such that $$(T_{k_2}^{l_3}\circ T_{k_1})(\kappa^2(x))\in [\Pi_{\beta}((k_1,k_2^n)^{\infty}),\Pi_{\beta}((k_2,k_1,k_2^{n-1})^{\infty})].$$ Let $\kappa^3=(\kappa^2,T_{k_1},T_{k_2}^{l_3}).$ Then by \eqref{checke} and \eqref{level2 bound} we have 
\begin{equation}
\label{checkg}
\Big| |(\kappa_i^3)_{i=1}^N|_{k_1}-p_{k_1}N\Big|\leq  2n'+L+2\textrm{ for all }1\leq N\leq |\kappa^3|.
\end{equation} Moreover, we also have 
\begin{equation}
\label{checkn}
-n'\leq |\kappa^3|_{k_1}-p_{k_1}|\kappa_3|\leq n'+L+1.
\end{equation}
If $|\kappa^3|_{k_1}-p_{k_1}|\kappa^3|\leq 0$ we stop and let $\tau^2=\kappa^3$. If $|\kappa^3|_{k_1}-p_{k_1}|\kappa^3|$ is positive we can repeat this process. We iteratively define $\kappa^4,\kappa^5,\ldots,$ etc. We stop if 
\begin{equation}
\label{signchangec}
|\kappa^j|_{k_1}-p_{k_1}|\kappa^j|\leq 0
\end{equation} for some $j$. Assume $j^*$ is the smallest such $j$ such that \eqref{signchangec} occurs. We then let $\tau^2=\kappa^{j^*}$. Importantly, analogues of \eqref{checkg} and \eqref{checkn} hold for each intermediate $\kappa$ term. Consequently
\begin{equation*}
\Big| |(\tau_i^2)_{i=1}^N|_{k_1}-p_{k_1}N\Big|\leq  2n'+L+2\textrm{ for all }1\leq N\leq |\tau^2|
\end{equation*} and 
\begin{equation*}
-n'\leq |\tau^2|_{k_1}-p_{k_1}|\tau^2|\leq 0.
\end{equation*} We also remark that at this point 
$$\tau^2(x)\in [\Pi_{\beta}((k_1,k_2^n)^{\infty}),\Pi_{\beta}((k_2,k_1,k_2^{n-1})^{\infty})].$$

 If there does not exist a $j$ such that \eqref{signchangec} occurs, then we let $\tau\in\{T_{k_1},T_{k_2}\}^{\mathbb{N}}$ be the infinite sequence we attain as the limit of the $\kappa^j$. Since each $\kappa^{j}$ is a prefix of $\kappa^{j'}$ for any $j'>j$ the infinite sequence $\tau$ is well defined. In this case the following holds for each $j\in\mathbb{N}$
\begin{equation*}
\Big| |(\kappa_i^j)_{i=1}^N|_{k_1}-p_{k_1}N\Big|\leq  2n'+L+2\textrm{ for all }1\leq N\leq |\kappa^j|.
\end{equation*} Consequently,
\begin{equation*}
\Big| |(\tau_i)_{i=1}^N|_{k_1}-p_{k_1}N\Big|\leq  2n'+L+2\textrm{ for all }N\in\mathbb{N}.
\end{equation*} So in this case $\tau$ would satisfy property $(2).$
\\

\noindent \textbf{Step $j+1$.} Suppose we have constructed $\tau^j$ such that 
$$\Big| |(\tau_i^j)_{i=1}^N|_{k_1}-p_{k_1}N\Big|\leq  2n'+L+2\textrm{ for all }1\leq N\leq |\tau^j|.$$ Moreover, assume $\tau^j$ satisfies
\begin{equation}
\label{casea}
\tau^j(x)\in [\Pi_{\beta}((k_1,k_2,k_1^{n-1})^{\infty}),\Pi_{\beta}((k_2,k_1^n)^{\infty})] \textrm{ and }0\leq |\tau^j|_{k_1}-p_{k_1}|\tau^i|\leq \max\{L,n'\}
\end{equation} or
\begin{equation}
\label{caseb}\tau^j(x)\in [\Pi_{\beta}((k_1,k_2^n)^{\infty}),\Pi_{\beta}((k_2,k_1,k_2^{n-1})^{\infty})] \textrm{ and }-\max\{L,n'\}\leq|\tau^j|_{k_1}-p_{k_1}|\tau^i|\leq 0.
\end{equation} Note that at the end of step $2$ the sequence $\tau^2$ we've constructed satisfies the first condition and \eqref{caseb}.

Assume \eqref{caseb} holds, the case where \eqref{casea} holds is handled similarly. We now essentially repeat the argument given in Step $2$. By Lemma \ref{third step} we can map $\tau^j(x)$ into $[\Pi_{\beta}((k_1,k_2,k_1^{n-1})^{\infty}),\Pi_{\beta}((k_2,k_1^n)^{\infty})]$ using at most $L$ maps. We then repeatedly map this image of $\tau^j(x)$ back into $[\Pi_{\beta}((k_1,k_2,k_1^{n-1})^{\infty}),\Pi_{\beta}((k_2,k_1^n)^{\infty})]$ using maps of the form $T_{k_1}^{l}\circ T_{k_2}$ where $n\leq l\leq n'$. We stop if we observe a change of sign. If we observe a change of sign the sequence we will have constructed is our $\tau^{j+1}$. It can be shown that this $\tau^{j+1}$ will then satisfy 
\begin{equation}
\label{important}
\Big| |(\tau_i^{j+1})_{i=1}^N|_{k_1}-p_{k_1}N\Big|\leq  2n'+L+2\textrm{ for all }1\leq N\leq |\tau^{j+1}|,
\end{equation} $$\tau^{j+1}(x)\in[\Pi_{\beta}((k_1,k_2,k_1^{n-1})^{\infty}),\Pi_{\beta}((k_2,k_1^n)^{\infty})],$$and  $$0\leq|\tau^j|_{k_1}-p_{k_1}|\tau^{j+1}|\leq \max\{L,n'\}$$ as required. If we never observe a sign change, then the infinite sequence we attain is our desired $\tau$. It satisfies $$\Big| |(\tau_i)_{i=1}^N|_{k_1}-p_{k_1}N\Big|\leq  2n'+L+2\textrm{ for all }N\in\mathbb{N}.$$
Clearly we can either repeat step $j+1$ indefinitely, in which case the infinite limit of the $\tau^j$'s will satisfy property $(2)$ by \eqref{important}, or at some point $\tau^j$ does not give rise to an $\tau^{j+1}$. This is the case where we do not observe a sign change. In this case, the infinite sequence we would obtain by repeatedly applying either $T_{k_2}^l\circ T_{k_1}$ or $T_{k_1}^l\circ T_{k_2}$   satisfies property $(2)$. In either case the desired $\tau$ exists and we can take $C=2n'+L+2.$

\end{proof}
Proposition \ref{important prop} has the useful consequence that for any $N\in\mathbb{N},$ the sequence $(\tau_1,\ldots,\tau_N)$ has $k_1$ frequency approximately $p_{k_1}$ and $k_2$ frequency approximately $p_{k_2}$. We will use this fact in the proof of Theorem \ref{Main theorem}.

Given $\epsilon>0$ recall that $$\Delta_{M,\epsilon}=\Big\{(p_k)_{k=0}^{M}\in \mathbb{R}^{M+1}:0\leq p_k\leq 1- \epsilon,\, \sum_{k=0}^M p_k =1\Big\}.$$ Clearly $\Delta_{M,\epsilon}$ is a compact, convex subset of $\mathbb{R}^{M+1}$. It is a simple exercise to check that it's extremal points are all the vectors of the form $\textbf{q}_{\epsilon,k_1,k_2}=(q_{\epsilon,0},\ldots,q_{\epsilon,M})$ where all entries are zero apart from $q_{\epsilon,k_1}=1-\epsilon,$ $q_{\epsilon,k_2}=\epsilon.$

Given $k_1,k_2\in\{0,\ldots,M\}$ such that $k_1\neq k_2,$ not necessarily $k_1<k_2$, let $\textbf{v}_{n,k_1,k_2}=(v_0,\ldots,v_M)$ where all entries are zero apart from $v_{k_1}=\frac{n}{n+1}$ and $v_{k_2}=\frac{1}{n+1}.$ In the following we denote the convex hull of a finite set of vectors by $\textrm{Conv}(\cdot).$

\begin{proposition}
	\label{convex prop}
For any $n\in\mathbb{N}$ such that $\frac{1}{n+1}\leq\epsilon$ we have 
$$\Delta_{M,\epsilon}\subseteq \textrm{Conv}(\{\textbf{v}_{n,k_1,k_2}\}_{k_1,k_2}).$$
\end{proposition}
\begin{proof}
By the Krein-Milman theorem (see \cite{DS}) it suffices to check that the extremal points of $\Delta_{M,\epsilon}$ are in the convex hull of $\{\textbf{v}_{n,k_1,k_2}\}_{k_1,k_2}$. However, for any $\frac{1}{n+1}\leq\epsilon$ we clearly have that $\textbf{q}_{\epsilon,k_1,k_2}$ is a convex combination of $\textbf{v}_{n,k_1,k_2}$ and $\textbf{v}_{n,k_2,k_1}.$ 
\end{proof}

Equipped with Proposition \ref{important prop} and Proposition \ref{convex prop} we are now in a position to prove Theorem \ref{Main theorem}. Before giving our proof we give an outline of our argument. Suppose $\beta\in(1,\beta_n)$ and $\textbf{p}\in \Delta_{M,\frac{1}{n+1}}$. By Proposition \ref{important prop} we know that for each $k_1\neq k_2$ we can construct finite sequences of maps $\tau_{k_1,k_2}$ of an arbitrary length with frequencies equal to $\frac{n}{n+1}$ and $\frac{1}{n+1}$ up to a bounded error term. By Proposition \ref{convex prop} we know that  $\textbf{p}\in \textrm{Conv}(\{\textbf{v}_{n,k_1,k_2}\}_{k_1,k_2})$. This proposition guarantees the existence of weights that may be used to construct $\textbf{p}$ from the frequencies of the different $\tau_{k_1,k_2}$. The problem is that we cannot freely concatenate the $\tau_{k_1,k_2}.$ We have to travel between the $D_{\beta,k_1,k_2,n}$ which introduces an error. However, by Lemma \ref{free movement} this error can always be bounded. Consequently by taking repeatedly larger $\tau_{k_1,k_2}$ this error becomes progressively more negligible, meaning the limiting sequence we construct will achieve the desired frequency $\textbf{p}$.

\begin{proof}[Proof of Theorem \ref{Main theorem}]
Fix $n\in\mathbb{N}$ and $\textbf{p}=(p_0,\ldots,p_M)\in\Delta_{M,\frac{1}{n+1}}.$ To prove Theorem \ref{Main theorem} it suffices to show that for any $\beta\in(1,\beta_n)$ and $x\in(0,\frac{M}{\beta-1})$ we have $\textbf{p}\in \Delta_{M,\beta}(x)$. 

By Proposition \ref{convex prop} there exists $r_{k_1,k_2}$ for all $k_1\neq k_2$ such that $r_{k_1,k_2}\geq 0,$ $$\sum_{(k_1,k_2)\atop k_1\neq k_2}r_{k_1,k_2}=1$$ and  
\begin{equation}
\label{vectorwise}
\textbf{p}=\sum_{(k_1,k_2)\atop k_1\neq k_2} r_{k_1,k_2}\textbf{v}_{n,k_1,k_2}.
\end{equation}
Evaluating \eqref{vectorwise} and using the definition of $\textbf{v}_{n,k_1,k_2}$ we have the following componentwise formula,
\begin{equation}
\label{componentwise}
p_{k}=\sum_{k_2=0\atop k_2\neq k}^M\frac{n\cdot r_{k,k_2}}{n+1}+\sum_{k_1=0\atop k_1\neq k}^M\frac{r_{k_1,k}}{n+1}
\end{equation} for any $k\in\{0,\ldots,M\}.$
Labelling terms we can write $\{(k_1,k_2):k_1\neq k_2\}=\{(k_1^p,k_2^p)\}_{p=1}^{M(M+1)}.$ 
Let us now fix a sequence of natural numbers $(N_i)$ such that $N_i\to \infty,$
\begin{equation}
\label{slowgrowth}
\lim_{j\to\infty}\frac{\sum_{i=1}^{j+1}N_i}{\sum_{i=1}^{j}N_i}=1
\end{equation}and 
\begin{equation}
\label{slowgrowth2}
\lim_{j\to\infty}\frac{j}{\sum_{i=1}^{j}N_i}=0.
\end{equation}For example one could take $N_i=i^2$. To each $(k_1^p,k_2^p)$ and $i\in\mathbb{N}$ we associate 
\begin{equation}
\label{intermediate integers}
N_{i,k_1^p,k_2^p}:=\lfloor N_i\cdot r_{k_1^p,k_2^p}\rfloor.
\end{equation}Where $\lfloor\cdot \rfloor$ denotes the integer part. We now devise an algorithm to construct the desired $\alpha$.
\\

\noindent \textbf{Step 1.} Without loss of generality we may assume that $x\in D_{\beta,k_1^1,k_2^1,n}.$ By Proposition \ref{important prop} we can construct $\tau^{1,1}$ such that $\tau^{1,1}(x) \in D_{\beta,k_1^1,k_2^1,n},$ $\tau^{1,1}\in\{T_{k_1^1},T_{k_2^1}\}^{N_{1,k_1^1,k_2^1}},$ $$\Big| |(\tau^{1,1})|_{k_1^1}-\frac{n\cdot N_{1,k_1^1,k_2^1}}{n+1}\Big|\leq C$$ and $$\Big| |(\tau^{1,1})|_{k_2^1}-\frac{N_{1,k_1^1,k_2^1}}{n+1}\Big|\leq C.$$ Since $\tau^{1,1}(x) \in D_{\beta,k_1^1,k_2^1,n}$ we can apply Lemma \ref{free movement} to $\tau^{1,1}(x)$ to obtain $\eta^{1,1}$ such that $|\eta^{1,1}|\leq L$ and $$(\eta^{1,1} \circ \tau^{1,1})(x)\in D_{\beta,k_1^2,k_2^2,n}.$$ Let $\psi^{1,1}:=(\tau^{1,1},\eta^{1,1}).$ 

Since $\psi^{1,1}(x) \in D_{\beta,k_1^2,k_2^2,n},$ we know by Proposition \ref{important prop} that there exists $\tau^{1,2}$ such that $(\tau^{1,2}\circ\psi^{1,1})(x)\in D_{\beta,k_1^2,k_2^2,n},$ $\tau^{1,2}\in \{T_{k_1^2},T_{k_2^2}\}^{N_{1,k_1^2,k_2^2}},$ $$\Big| |(\tau^{1,2})|_{k_1^2}-\frac{n\cdot N_{1,k_1^2,k_2^2}}{n+1}\Big|\leq C$$ and $$\Big| |(\tau^{1,2})|_{k_2^2}-\frac{N_{1,k_1^2,k_2^2}}{n+1}\Big|\leq C.$$ Since $(\tau^{1,2}\circ\psi^{1,1})(x)\in D_{\beta,k_1^2,k_2^2,n}$ we can apply Lemma \ref{free movement} to $(\tau^{1,2}\circ\psi^{1,1})(x)$ to obtain $\eta^{1,2}$ such that $|\eta^{1,2}|\leq L$ and $(\eta^{1,2}\circ\tau^{1,2}\circ\psi^{1,1})(x)\in D_{\beta,k_1^3,k_2^3,n}.$ Let $\psi^{1,2}:=(\psi^{1,1},\tau^{1,2},\eta^{1,2}).$ 

We repeat this argument until eventually we obtain a sequence $\psi^{1}$ of the form
$$\psi^{1}=(\tau^{1,1},\eta^{1,1},\cdots, \tau^{1,M(M+1)},\eta^{1,M(M+1)}),$$ such that $|\eta^{1,p}|\leq L$ for all $1\leq p\leq M(M+1),$ $\tau^{1,p}\in \{T_{k_1^p},T_{k_2^p}\}^{N_{1,k_1^p,k_2^p}}$  for all $1\leq p\leq M(M+1),$ $\psi^{1}(x)\in D_{\beta,k_1^1,k_2^1,n},$  
$$\Big| |(\tau^{1,p})|_{k_1^p}-\frac{n\cdot N_{1,k_1^p,k_2^p}}{n+1}\Big|\leq C$$and $$\Big| |(\tau^{1,p})|_{k_2^p}-\frac{N_{1,k_1^p,k_2^p}}{n+1}\Big|\leq C$$ for all $1\leq p\leq M(M+1)$. Let $\alpha^1=\psi^1.$
\\

\noindent \textbf{Step $j+1$.} Suppose we have constructed $(\psi^{i})_{i=1}^j$ and $(\alpha^{i})_{i=1}^j$ which satisfy
\begin{enumerate}
	\item For each $1\leq i\leq j$ $$\alpha^i=(\psi^1,\ldots,\psi^i)$$ and  $\alpha^{i}(x)\in D_{\beta,k_1^1,k_2^1,n}.$ 
	\item For each $1\leq i \leq j$ $$\psi^i=(\tau^{i,1},\eta^{i,1},\cdots, \tau^{i,M(M+1)},\eta^{i,M(M+1)})).$$ 
	\item For each $1\leq i\leq j$ and $1\leq p\leq M(M+1)$ we have $|\eta^{i,p}|\leq L.$ 
	\item For each $1\leq i\leq j$ and $1\leq p\leq M(M+1)$ we have $\tau^{i,p}\in \{T_{k_1^p},T_{k_2^p}\}^{N_{i,k_1^p,k_2^p}}.$ 
	\item For each $1\leq i\leq j$ and $1\leq p\leq M(M+1)$ we have 
	\begin{equation}
	\label{frequencybounda}
	\Big| |(\tau^{i,p})|_{k_1^p}-\frac{n\cdot N_{i,k_1^p,k_2^p}}{n+1}\Big|\leq C
	\end{equation}
	and  
	\begin{equation}
	\label{frequencyboundb}
	\Big||(\tau^{i,p})|_{k_2^p}-\frac{N_{i,k_1^p,k_2^p}}{n+1}\Big|\leq C.
	\end{equation} 
\end{enumerate}  
Repeating the argument given in step $1$ with $x$ replaced by $\alpha^j(x)$ we obtain a sequence $$\psi^{j+1}:=(\tau^{j+1,1},\eta^{j+1,1},\cdots, \tau^{j+1,M(M+1)},\eta^{j+1,M(M+1)})$$ such that $\psi^{j+1,N}(\alpha^{j}(x))\in D_{\beta,k_1^1,k_2^1,n},$ $|\eta^{j+1,p}|\leq L$ for all $1\leq p\leq M(M+1),$ $\tau^{j+1,p}\in \{T_{k_1^p},T_{k_2^p}\}^{N_{j+1,k_1^p,k_2^p}}$ for all $1\leq p\leq M(M+1),$  
$$\Big| |(\tau^{j+1,p})|_{k_1^p}-\frac{n\cdot N_{j+1,k_1^p,k_2^p}}{n+1}\Big|\leq C$$
 and 
$$\Big||(\tau^{1,p})|_{k_2^p}-\frac{N_{j+1,k_1^p,k_2^p}}{n+1}\Big|\leq C$$
 for all $1\leq p\leq M(M+1)$. We then let $$\alpha^{j+1}:=(\alpha^{j},\psi^{j+1}).$$ Clearly $\psi^{j+1}$ and $\alpha^{j+1}$ satisfy properties $(1)-(5)$. This completes our inductive step.
\\

By property $(1)$ above it follows that $\alpha^{j}$ is prefix of $\alpha^{j'}$ for any $j'>j$, so the limiting infinite sequence $\alpha$ is well defined. Clearly $\alpha\in \Omega_{\beta,M}(x)$ by property $(1)$. By Lemma \ref{Bijection lemma} to prove our theorem it remains to show that $\alpha$ satisfies the required digit frequency properties. 

Observe that for any $i\in\mathbb{N}$ we have
\begin{equation}
\label{length bounda}
\Big||\psi^{i}|-\sum_{p=1}^{M(M+1)} N_{i,k_1^p,k_2^p}\Big|\leq LM(M+1).
\end{equation}Equation \eqref{length bounda} follows from properties $(2),(3),$ and $(4)$. It follows from \eqref{intermediate integers}, \eqref{length bounda}, and the fact $\sum_{p=1}^{M(M+1)} r_{k_1^p,k_2^p}=1$ that 
\begin{equation}
\label{length boundb}
\Big||\psi^{i}|-N_i\Big| \leq (L+1)M(M+1)
\end{equation} for all $i\in\mathbb{N}$. What is more, for any $i\in\mathbb{N}$ and $k\in\{0,\ldots,M\},$ it follows from properties $(2)$ and $(4),$ and equations \eqref{frequencybounda} and \eqref{frequencyboundb} that
$$\Big||\psi^{i}|_{k}-\sum_{k_2=0\atop k_2\neq k}^M\frac{n\cdot N_{i,k,k_2}}{n+1}-\sum_{k_1=0\atop k_1\neq k}^M\frac{N_{i,k_1,k}}{n+1}\Big|\leq (L+C)M(M+1).$$Using \eqref{intermediate integers} we obtain
$$\Big||\psi^{i}|_{k}-\sum_{k_2=0\atop k_1\neq k}^M\frac{n\cdot N_i\cdot r_{k,k_2}}{n+1}-\sum_{k_1=0\atop k_1\neq k}^M\frac{N_i\cdot r_{k_1,k}}{n+1}\Big|\leq (L+C+1)M(M+1).$$
Applying \eqref{componentwise} we see that 
\begin{equation}
\label{finalstep}
\Big||\psi^{i}|_{k}-N_i \cdot p_k\Big|\leq (L+C+1)M(M+1).
\end{equation}

For any $n\in\mathbb{N}$ there exists $j\in\mathbb{N}$ such that $|\alpha^{j}| \leq n<|\alpha^{j+1}|.$ Therefore for any $k\in\{0,\ldots,M\}$ we have
\begin{align*}
\frac{|(\alpha_l)_{l=1}^n |_{k}}{n}&\leq \frac{|\alpha^{j+1}|_{k}}{|\alpha^{j}|}\\
&\leq \frac{|\alpha^{j+1}|_{k}}{\sum_{i=1}^jN_i-j(L+1)M(M+1)}&\,\textrm{By }\eqref{length boundb}\\
&\leq \frac{p_k\sum_{i=1}^{j+1}N_i+(j+1)(L+C+1)M(M+1)}{\sum_{i=1}^jN_i-j(L+1)M(M+1)}&\, \textrm{By }\eqref{finalstep}.
\end{align*}
It follows from the above, \eqref{slowgrowth}, and \eqref{slowgrowth2} that $$\limsup_{n\to\infty} \frac{|(\alpha_l)_{l=1}^n |_{k}}{n}\leq p_k.$$ Similarly, one can show that $$\liminf_{n\to\infty}\frac{|(\alpha_l)_{l=1}^n |_{k}}{n}\geq p_k.$$ Consequently $$\lim_{n\to\infty} \frac{|(\alpha_l)_{l=1}^n |_{k}}{n}= p_k.$$ Since $k$ was arbitrary this completes our proof.

\end{proof}

\begin{proof}[Proof of Theorem \ref{noexistence}]
The proof of Theorem \ref{noexistence} is an adaptation of the proof of Theorem \ref{Main theorem}. As such we just provide an outline and leave the details to the interested reader. Let $D$ and $(p_k)_{k\in D^c}$ be as in the statement of Theorem \ref{noexistence}. Under the assumptions of this theorem there exists $\textbf{q}\in \Delta_{M,\frac{1}{n+1}}$ and $\textbf{q}'\in \Delta_{M,\frac{1}{n+1}}$ such that $q_k\neq q_k'$ for all $k\in D$, and $q_k=q_k'=p_k$ for all $k\in D^c$. 

Given an $x\in(0,\frac{M}{\beta-1})$ we can carefully go through the argument given in the proof of Theorem \ref{Main theorem} to construct an $\alpha\in \Omega_{\beta,M}(x)$ satisfying the following three properties: 
\begin{enumerate}
	\item There exists a sequence $(N_p)$ such that $$\lim_{p\to\infty}\frac{|(\alpha_l)_{l=1}^{N_p}|_k}{N_p}=q_k$$ for all $k\in D.$
	\item There exists a sequence $(N_j)$ such that $$\lim_{j\to\infty}\frac{|(\alpha_l)_{l=1}^{N_j}|_k}{N_j}=q_k'$$ for all $k\in D.$
	\item For all $k\in D^c$ we have $$\lim_{n\to\infty}\frac{|(\alpha_l)_{l=1}^{n}|_k}{n}=p_k.$$
\end{enumerate}
Since $q_k\neq q_k'$ for all $k\in D$ it follows that $\alpha$ satisfies the required digit frequency properties. 

To construct the $\alpha\in \Omega_{\beta,M}(x)$ described above one proceeds initially as in the proof of Theorem \ref{Main theorem} as if we were trying to build an expansion with digit frequencies described by the vector $\textbf{q}$. Once we have a sufficiently good approximation to $\textbf{q}$ we change our algorithm to construct an expansion with digit frequencies described by $\textbf{q}',$ then once we have a sufficiently good approximation to $\textbf{q}'$ we switch back to $\textbf{q}$ and so on.  
\end{proof}
\section{Proof of Theorem \ref{asymptotics}}
In this section we prove Theorem \ref{asymptotics}. We split the proof into the following two propositions.
\begin{proposition}
	\label{propa}
	For all $n\in\mathbb{N}$ we have $$1+\frac{\log n -\log \log n}{n}< \beta_n.$$
\end{proposition}
\begin{proof}
Recall that $\beta_n$ is the unique solution in $(1,2)$ of the polynomial $$f_n(x)=x^{n+1}-x^n-1.$$ Evaluating $f_n$ at $1+\frac{\log n -\log \log n}{n}$ we obtain
\begin{align*}
f_n\Big(1+\frac{\log n -\log \log n}{n}\Big)&=\Big(1+\frac{\log n -\log \log n}{n}\Big)^{n+1}-\Big(1+\frac{\log n -\log \log n}{n}\Big)^{n}-1\\
&=\Big(1+\frac{\log n -\log \log n}{n}\Big)^{n}\Big(1+\frac{\log n-\log \log n}{n}-1\Big)-1\\
&\leq e^{\log n -\log \log n}\cdot \frac{\log n-\log \log n}{n}-1\\
	&=\frac{n}{\log n}\cdot \frac{\log n-\log \log n}{n}-1\\
	&=1-\frac{\log \log n}{\log n}-1\\
	&< 0.
\end{align*} To obtain the third line above we used the fact that $(1+\frac{x}{n})^n\leq e^x$ for all $x\geq 0$. Since $f_n(1)<0$ and $f_n'(x)>1$ for all $x\geq 1$ it follows from the above that $1+\frac{\log n -\log \log n}{n}< \beta_n.$
\end{proof}

\begin{proposition}
	\label{propb}
	For any $M\in\mathbb{N}$ we have $$\beta_{M,n}^*\leq 1 + \frac{\log (n+1)}{n+1}+\mathcal{O}\big(\frac{1+\log M}{n+1}\big).$$
\end{proposition}
\begin{proof}
Fix $n\in\mathbb{N}$ and $M\in\mathbb{N}$. For any $\epsilon>0,$ $N\in\mathbb{N}$ and $k\in\{0,\ldots,M\},$ let $$S_{\epsilon,N,k}:=\Big\{(a_i)_{i=1}^N\in\{0,\ldots,M\}^N:\frac{\#\{1\leq i\leq N:a_i=k\}}{N}\geq \frac{n}{n+1}-\epsilon\Big\}.$$ Using a large deviation result of Hoeffding, in particular see Theorem 1 from \cite{Hoe}, one can show that
\begin{equation}
\label{bound}
\# S_{\epsilon,N,k}\leq \exp \Big(N\Big(\big(\frac{1}{n+1}+\epsilon\big)\log M -\big(\frac{n}{n+1}-\epsilon\big)\log\big(\frac{n}{n+1}-\epsilon\big)- \big(\frac{1}{n+1}+\epsilon\big)\log \big(\frac{1}{n+1}+\epsilon\big)\Big)\Big).
\end{equation}Suppose now that 
\begin{equation}
\label{beta bound}
\exp \Big(\big(\frac{1}{n+1}+\epsilon\big)\log M -\big(\frac{n}{n+1}-\epsilon\big)\log\big(\frac{n}{n+1}-\epsilon\big)- \big(\frac{1}{n+1}+\epsilon\big)\log \big(\frac{1}{n+1}+\epsilon\big)\Big)<\beta.
\end{equation} We will show that $$\Big\{x\in\Big[0,\frac{M}{\beta-1}\Big]:\Delta_{M,\frac{1}{n+1}}\subseteq \Delta_{M,\beta}(x)\Big\}$$ has Lebesgue measure zero for all $\beta$ satisfying \eqref{beta bound}. 
We start by remarking that 
\begin{equation}
\label{Borel Cantelli}
\Big\{x\in\Big[0,\frac{M}{\beta-1}\Big]:\Delta_{M,\frac{1}{n+1}}\subseteq \Delta_{M,\beta}(x)\Big\}\subseteq \bigcap_{L=1}^{\infty}\bigcup_{N=L}^{\infty}\bigcup_{(a_i)\in S_{\epsilon,N,k}} \Big[\sum_{i=1}^{N}\frac{a_i}{\beta^i},\sum_{i=1}^{N}\frac{a_i}{\beta^i}+\frac{M}{\beta^N(\beta-1)}\Big]
\end{equation} for any $k \in\{0,\ldots,M\}$ and $\epsilon>0$. The right hand side of \eqref{Borel Cantelli} is a limsup set, so by the Borel Cantelli lemma to prove that the set on the right hand side of has zero Lebesgue measure, and therefore the left hand side has zero Lebesgue measure, it suffices to show that 
$$\sum_{N=1}^{\infty}\sum_{(a_i)\in S_{\epsilon,N,K}} \mathcal{L}\Big(\Big[\sum_{i=1}^{N}\frac{a_i}{\beta^i},\sum_{i=1}^{N}\frac{a_i}{\beta^i}+\frac{M}{\beta^N(\beta-1)}\Big]\Big)<\infty.$$
Where here and throughout $\mathcal{L}$ denotes the Lebesgue measure. This bound follows immediately since  
$$\sum_{N=1}^{\infty}\sum_{(a_i)\in S_{\epsilon,N,K}} \mathcal{L}\Big(\Big[\sum_{i=1}^{N}\frac{a_i}{\beta^i},\sum_{i=1}^{N}\frac{a_i}{\beta^i}+\frac{M}{\beta^N(\beta-1)}\Big]\Big) =\sum_{N=1}^{\infty}\frac{M\cdot \# S_{\epsilon,N,k}}{\beta^N(\beta-1)}<\infty.$$Where the final inequality holds by \eqref{bound} and \eqref{beta bound}.
It follows that $$\Big\{x\in\Big[0,\frac{M}{\beta-1}\Big]:\Delta_{M,\frac{1}{n+1}}\subseteq \Delta_{M,\beta}(x)\Big\}$$ has Lebesgue measure zero. So we may conclude that 
$$\beta_{M,n}^*\leq \exp \Big(\big(\frac{1}{n+1}+\epsilon\big)\log M -\big(\frac{n}{n+1}-\epsilon\big)\log\big(\frac{n}{n+1}-\epsilon\big)- \big(\frac{1}{n+1}+\epsilon\big)\log \big(\frac{1}{n+1}+\epsilon\big)\Big).$$ Since $\epsilon$ is arbitrary it follows that
\begin{equation}
\label{lose epsilon}
\beta_{M,n}^*\leq \exp \Big(\frac{\log M}{n+1} -\frac{n}{n+1}\log \frac{n}{n+1}- \frac{1}{n+1}\log \frac{1}{n+1}\Big).
\end{equation} Observe that $$-\frac{n}{n+1}\log \frac{n}{n+1}\leq \frac{1}{n+1}$$ for all $n\in \mathbb{N}$. Substituting this bound into \eqref{lose epsilon} we obtain  $$\beta_{M,n}^{*}\leq \exp\Big(\frac{\log( n+1)}{n+1}+ \frac{1+\log M}{n+1}\Big).$$ Using the expansion of $e^x=1+x+\mathcal{O}(x^2)$ we obtain 
\begin{align*}
\beta_{M,n}^{*}&\leq 1 + \frac{\log (n+1)}{n+1}+ \frac{1+\log M}{n+1}+\mathcal{O}\Big(\big(\frac{\log (n+1)}{n+1}+ \frac{1+\log M}{n+1}\big)^2\Big)\\
& = 1 + \frac{\log (n+1)}{n+1}+\mathcal{O}\Big(\frac{1+\log M}{n+1}\Big).
\end{align*}
\end{proof}Combining Proposition \ref{propa} and Proposition \ref{propb} we deduce Theorem \ref{asymptotics}.

\section{Applications to biased Bernoulli convolutions}
Given $M\in\mathbb{N}$ and $\textbf{q}=(q_0,\ldots, q_M)\in \Delta_M,$ one can define a product measure $\mathbb{P}_{\textbf{q}}:=\prod_{1}^{\infty}\{q_0,\ldots,q_M\}$ on $\{0,\ldots,M\}^{\mathbb{N}}.$ Given a $\beta>1$, one then defines the Bernoulli convolution associated to $\textbf{q}$ via the equation $$\mu_{\textbf{q}}(E)= \mathbb{P}_{\textbf{q}}\Big(\Big\{(a_i)\in \{0,\ldots,M\}^{\mathbb{N}}:\sum_{i=1}^{\infty}\frac{a_i}{\beta^i}\in E\Big\}\Big),$$ where $E$ is an arbitrary Borel subset of $\mathbb{R}$. Bernoulli convolutions have been studied since the 1930's. They've connections with algebraic numbers, dynamical systems, and fractal geometry. The fundamental question surrounding Bernoulli convolutions is to determine those $\beta\in(1,M+1]$ and $\textbf{q}\in \Delta_M$ such that $\mu_{\textbf{q}}$ is absolutely continuous with respect to the Lebesgue measure. For more on this class of measures we refer the reader to the surveys of Peres, Schlag, and Solomyak \cite{Sol}, and Varju \cite{Var}. 

Given $x\in I_{\beta,M}$ one defines the lower and upper local dimensions of $\mu_{\textbf{q}}$ at $x$ to be $$\underline{d}_{\mu_{\textbf{q}}}(x)=\liminf_{r\to 0}\frac{\log \mu_{\textbf{q}}([x-r,x+r])}{\log r} \textrm{ and }\overline{d}_{\mu_{\textbf{q}}}(x)=\limsup_{r\to 0}\frac{\log \mu_{\textbf{q}}([x-r,x+r])}{\log r}$$ respectively. These quantities describe how the measure $\mu_{\textbf{q}}$ scales locally around a point $x$. We remark that when it comes to calculating $\underline{d}_{\mu_{\textbf{q}}}$ and $\overline{d}_{\mu_{\textbf{q}}}$ it suffices to consider sequences of the form $r_n:=\frac{M}{\beta^N(\beta-1)}$.

\begin{proposition}
	\label{Bernoulli bound}
Let $x\in I_{\beta,M}$ and $\textbf{p}\in \Delta_{M}$. Suppose there exists $(a_i)\in \Sigma_{\beta,M}(x)$ satisfying  $\text{freq}_{k}(a_i)=p_k$ for all $0\leq k\leq M.$ Then $$\overline{d}_{\mu_{\textbf{q}}}(x)\leq  -\sum_{k=0}^M\frac{p_k \log q_k}{\log \beta}.$$
\end{proposition}

\begin{proof}
For any $x\in I_{\beta,M}$ and $(a_i)\in \Sigma_{\beta,M}(x)$ we have 
$$\mu_{\textbf{q}}\Big(\Big[\sum_{i=1}^{N}\frac{a_i}{\beta^i},\sum_{i=1}^{N}\frac{a_i}{\beta^i}+\frac{M}{\beta^N(\beta-1)}\Big]\Big)\leq \mu_{\textbf{q}}\Big(\Big[x-\frac{M}{\beta^{N}(\beta-1)},x+\frac{M}{\beta^N(\beta-1)}\Big]\Big).$$ In which case it follows from the definition of $\mu_{\textbf{q}}$ that $$\prod_{k=0}^M q_k^{\#\{1\leq i\leq N:a_i=k\}}\leq \mu_{\textbf{q}}\Big(\Big[x-\frac{M}{\beta^{N}(\beta-1)},x+\frac{M}{\beta^N(\beta-1)}\Big]\Big).$$Consequently, 
\begin{align*}
\limsup_{r\to 0}\frac{\log \mu_{\textbf{q}}([x-r,x+r])}{\log r}&=\limsup_{N\to \infty}\frac{\log \mu_{\textbf{q}}([x-\frac{M}{\beta^{N}(\beta-1)},x+\frac{M}{\beta^N(\beta-1)}])}{\log \frac{M}{\beta^N(\beta-1)}}\\ 
&\leq \limsup_{N\to\infty} \sum_{k=0}^M \frac{\#\{1\leq i \leq N:a_i=k\}\log q_k}{\log \frac{M}{\beta^N(\beta-1)}}\\
& = -\sum_{k=0}^M\frac{p_k \log q_k}{\log \beta}.
\end{align*}
\end{proof}
Given \textbf{q} as above we let $$q^*=\max_{k}q_k\textrm{ and } q^{**}=\max_{k:q_k\neq q^*}q_k.$$ The following result follows immediately from Theorem \ref{Main theorem} and Proposition \ref{Bernoulli bound}.

\begin{corollary}
	\label{biased bound}
Let $M\in\mathbb{N}$ and $\beta\in(1,\beta_n).$ Then for any $\textbf{q}$ as above and $x\in (0,\frac{M}{\beta-1})$ we have  $$\overline{d}_{\mu_{\textbf{q}}}(x)\leq -\frac{(n\log q^* + \log q^{**})}{(n+1)\log \beta}.$$
\end{corollary}As an example, choosing $M=2$, $\beta=1.28,$ and $\textbf{q}=(0.8,0.15,0.05),$ Corollary \ref{biased bound} can be applied with $n=5$ to show that
$$\overline{d}_{\mu_{\textbf{q}}}(x)\leq 2.034\ldots$$ for all $x\in (0,\frac{M}{\beta-1}).$

Similarly, combining Theorem \ref{simple theorem} and Proposition \ref{Bernoulli bound} we conclude the following statement. 

\begin{corollary}
	Let $M=1$ and $\beta\in(1,\beta_T].$ Then for any $\textbf{q}=(q_1,q_2)$ and $x\in (0,\frac{1}{\beta-1})$ we have  $$\overline{d}_{\mu_{\textbf{q}}}(x)\leq -\frac{(\log q_1 + \log q_2)}{2\log \beta}.$$
\end{corollary}

%\section{Concluding remarks}
%We conclude by posing a few questions and making some comments. In Theorem \ref{asymptotics} we obtain bounds for the quantity $\beta_{M,n}.$ It would be interesting to determine this value explicitly. However this seems like a difficult problem. A more modest problem would be to explicitly determine the dependence on $M$ in the decay rate $\beta^{*}_{n,M}-1$. We know by Corollary \ref{asymptotics cor} that the leading order term has no dependence on $M$. What if any dependence on $M$ is there in the asymptotics of $\beta^{*}_{n,M}-1.$

%In \cite{BakKong} it was asked whether there exists a sharp constant $\beta_{c}(M)$ such that for every $\beta\in(1,\beta_{c}(M))$ every $x\in (0,\frac{M}{\beta-1})$ has a simply normal $\beta$-expansion. We obtained a weak solution to this question in Corollary \ref{simple cor} where we prove that for any $M\in\mathbb{N}$ and $\beta\in(1,\frac{1+\sqrt{5}}{2}),$ every $x\in (0,\frac{M}{\beta-1})$ admits a simply normal $\beta$-expansion. The techniques used in this paper involved focusing on subintervals of $I_{\beta,M}$ parameterised by a pair $(k_1,k_2)$ and building part of a $\beta$-expansion using only $k_1$ and $k_2$. It seems to the author that this technique cannot be extended past $\frac{1+\sqrt{5}}{2}.$ 

\noindent \textbf{Acknowledgements.} This research was supported by the EPSRC grant EP/M001903/1. The author would like to thank Wolfgang Steiner for posing the question that led to this work, and Thomas Jordan for pointing out the applications to biased Bernoulli convolutions. Part of this work was completed whilst the author was visiting the Mittag-Leffler institute as part of the program ``Fractal geometry and Dynamics". The author thanks the organisers and staff for their support.


\begin{thebibliography}{1}
\bibitem{BakD} S. Baker, \textit{Digit frequencies and self-affine sets with non-empty interior,} arXiv:1701.06773 [math.DS].
\bibitem{BakG} S. Baker, \textit{Generalised golden ratios over integer alphabets,} Integers 14 (2014), Paper No. A15.
\bibitem{Bak2} S. Baker, \textit{On small bases which admit countably many expansions,} Journal of Number Theory 147 (2015), 515--532.
\bibitem{Bak3} S. Baker, \textit{The growth rate and dimension theory of beta-expansions,} Fund. Math. 219 (2012), no. 3, 271--285. 
\bibitem{BakKong} S. Baker, D. Kong, \textit{Numbers with simply normal $\beta$-expansions,} arXiv:1707.01013 [math.DS].
\bibitem{BakerSid} S. Baker, N. Sidorov, \textit{Expansions in non-integer bases: lower order revisited} Integers 14 (2014), Paper No. A57.
\bibitem{Bes} A. S. Besicovitch, \textit{On the sum of digits of real numbers represented in the dyadic system,} Math. Ann., 110 (1) (1935), 321--330.
\bibitem{Bor} E. Borel, \textit{Les probabilités dénombrables et leurs applications arithmétiques,} Rendiconti del Circolo Matematico di Palermo (1909), 27: 247–-271,
\bibitem{BCH} P. Boyland, A. de Carvalho, T. Hall, \textit{On digit frequencies in $\beta$-expansions.} Trans. Amer. Math. Soc. 368 (2016), no. 12, 8633-–8674.
\bibitem{DKK} K. Dajani, K. Jiang, T. Kempton, \textit{Self-affine sets with positive Lebesgue measure,} Indag. Math. 25 (2014), 774–-784.
\bibitem{DS} N. Dunford, J. Schwartz, \textit{Linear Operators, Part 1}, Interscience Publishers Inc., New York, 1958. 
\bibitem{Egg} H. Eggleston, \textit{The fractional dimension of a set defined by decimal properties,} Quart. J. Math. Oxford Ser., 20 (1949), pp. 31–-36.
\bibitem{FLMW} A. Fan, L. Liao, J. Ma, B. Wang, \textit{Dimension of Besicovitch-Eggleston sets in countable symbolic space,} Nonlinearity 23 (2010), no. 5, 1185-–1197. 
\bibitem{Gun} C.S. G\"{u}nt\"{u}rk, \textit{Simultaneous and hybrid beta-encodings,} in Information Sciences
and Systems, 2008. CISS 2008. 42nd Annual Conference on, pages 743--748, 2008.
\bibitem{HS} K. Hare, N. Sidorov, \textit{On a family of self-affine sets: Topology, uniqueness, simultaneous expansions,} Ergodic Theory Dynam. Systems 37 (2017), no. 1, 193–-227. 
\bibitem{HocShm} M. Hochman, P. Shmerkin, \textit{Equidistribution from fractal measures,} Invent. Math. 202 (2015), no. 1, 427-–479. 
\bibitem{Hoe} W. Hoeffding, \emph{Probability inequalities for sums of bounded random variables,} J. Amer. Statist. Assoc. 58 1963 13–-30.
\bibitem{JSS} T. Jordan, P. Shmerkin, B. Solomyak, \textit{Multifractal structure of Bernoulli convolutions,} Math. Proc. Cambridge Philos. Soc., 151(3):521-–539, 2011.
\bibitem{Parry} W. Parry, \textit{On the $\beta$-expansions of real numbers}, Acta Math. Acad. Sci. Hung. {\bf 11} (1960) 401--416.
\bibitem{Renyi} A. R\'{e}nyi, \textit{Representations for real numbers and their ergodic properties}, Acta Math. Acad. Sci. Hung. {\bf 8} (1957) 477--493.
\bibitem{Sid2} N. Sidorov, \textit{Almost every number has a continuum of beta-expansions,} Amer. Math. Monthly 110 (2003), 838--842.
\bibitem{Sid3} N. Sidorov, \textit{Combinatorics of linear iterated function systems with overlaps}, Nonlinearity 20 (2007), no. 5, 1299–-1312. 
\bibitem{Sid1} N. Sidorov, \textit{Expansions in non-integer bases: lower, middle and top orders}, J. Number Th. {\bf 129} (2009), 741--754.
\bibitem{Sol} Y. Peres, W. Schlag, B. Solomyak, \textit{Sixty years of Bernoulli convolutions,} Fractal geometry and stochastics, II (Greifswald/Koserow, 1998), 39-–65, Progr. Probab., 46, Birkhäuser, Basel, 2000. 
\bibitem{Var} P. Varju, \textit{Recent progress on Bernoulli convolutions,} arXiv:1608.04210 [math.CA]
\end{thebibliography}
\end{document}